\documentclass[11pt,a4paper]{article}

\usepackage[all]{xy}
\usepackage{latexsym}
\usepackage[margin=3cm]{geometry}
\usepackage[dvipdfmx]{graphicx}

\usepackage{here}

\usepackage{mymacro}
\usepackage{tikz}
\usepackage{mathtools}

\renewcommand{\labelenumi}{$\mathrm{(\roman{enumi})}$}

\makeatletter
  
  \@addtoreset{equation}{section}
\makeatother

\title{Compact exact Lagrangian intersections in cotangent bundles via sheaf quantization}
\author{Yuichi Ike}
\date{\today}

\AtEndDocument{{
		\noindent Yuichi Ike 

		\noindent Fujitsu Laboratories Ltd., 4-1-1 Kamikodanaka, Nakahara-ku, Kawasaki, Kanagawa, 211-8588, Japan
		
		\noindent
		\textit{E-mail address}: \texttt{yuichi.ike.1990@gmail.com}, \texttt{ike.yuichi@fujitsu.com}
}}

\begin{document}
\maketitle

\begin{abstract}
	We show that the cardinality of the transverse intersection of two compact exact Lagrangian submanifolds in a cotangent bundle is bounded from below by the dimension of the Hom space of sheaf quantizations of the Lagrangians in Tamarkin's category.
	Our sheaf-theoretic method can also deal with clean and degenerate Lagrangian intersections.
\end{abstract}


\section{Introduction}

The study of Lagrangian intersections, especially intersections of exact Lagrangian submanifolds in cotangent bundles is an important problem in symplectic geometry.
In this paper, we study them using a method based on microlocal sheaf theory, more precisely, Tamarkin's category and Guillermou's sheaf quantization.
We state our main result and its corollary in \pref{subsec:intromain}.

\subsection{Applications of microlocal sheaf theory to symplectic geometry}

Microlocal sheaf theory was introduced and systematically developed by Kashiwara and Schapira~\cite{KS90}.
One of the key ingredients of the theory is the notion of microsupports of sheaves.
In the sequel, let $\bfk$ be a field.
Moreover, let $X$ be a $C^\infty$-manifold and denote by $\Db(X)$ the bounded derived category of sheaves of $\bfk$-vector spaces.
For an object $F \in \Db(X)$, its microsupport $\MS(F)$ is defined as the set of directions in which the cohomology of $F$ cannot be extended isomorphically.
The microsupport is a closed subset of the cotangent bundle $T^*X$ and conic, that is, invariant under the action of $\bR_{>0}$ on $T^*X$.

Tamarkin~\cite{Tamarkin} proposed a new approach to symplectic geometry, which is based on microlocal sheaf theory.
A sheaf whose microsupport coincides with a given conic Lagrangian submanifold of a cotangent bundle (outside the zero-section) is called a \emph{sheaf quantization} of the Lagrangian.
For a non-conic Lagrangian, one can consider a sheaf quantization by adding a variable and ``conifying" it.
Using sheaf quantizations, Tamarkin studied the intersections of particular Lagrangian submanifolds.
After his work, Guillermou--Kashiwara--Schapira~\cite{GKS} and Guillermou~\cite{Gu12, Gulec} proved the existence of sheaf quantizations of graphs of Hamiltonian isotopies and compact exact Lagrangian submanifolds in cotangent bundles, respectively.
See \pref{sec:quantandnondisp} for more details.
Note that sheaf-theoretic approaches to symplectic geometry also appeared in \cite{KO,NZ,Nad}.

\subsection{Our results}\label{subsec:intromain}

In this paper, we prove that the cardinality of the transverse intersection of compact exact Lagrangian submanifolds in a cotangent bundle is bounded from below by the dimension of the Hom space of sheaf quantizations of the Lagrangians in Tamarkin's category.
More generally, provided $\bfk=\bF_2=\bZ/2\bZ$, we show that a clean version of the estimate holds with ``cardinality" replaced by ``total $\bF_2$-Betti number".

In what follows, let $M$ be a compact connected $C^\infty$-manifold without boundary and denote by $T^*M$ its cotangent bundle.
We also denote by $(x;\xi)$ a local homogeneous coordinate system.
We regard $T^*M$ as an exact symplectic manifold equipped with the Liouville 1-form $\alpha=\langle \xi, dx \rangle$.
A submanifold $L$ of dimension $\dim M$ in $T^*M$ is said to be exact Lagrangian if $\alpha|_L$ is exact.
The main result of this paper is the following.
See \pref{sec:quantandnondisp} for the definitions of simple sheaf quantizations, $\cHom^\star$, and Tamarkin's category $\cT(M)$.

\begin{theorem}[{see \pref{thm:cleanineq}}]\label{thm:intromain}
	For $i=1,2$, let $L_i$ be a compact connected exact Lagrangian submanifold and $F_i \in \Db(M \times \bR)$ be a simple sheaf quantization associated with $L_i$ and a function $f_i \colon L_i \to \bR$ satisfying $df_i=\alpha|_{L_i}$ .
	Assume that $L_1$ and $L_2$ intersect cleanly, that is, $L_1 \cap L_2$ is a submanifold of $T^*M$ and $T_p(L_1 \cap L_2)=T_pL_1 \cap T_p L_2$ for any $p \in L_1 \cap L_2$.
	Let $L_1 \cap L_2 = \bigsqcup_{j=1}^n C_j$ be the decomposition into connected components and define $f_{21}(C_j)\coloneqq f_2(p)-f_1(p)$ for some $p \in C_j$ (independent of the choice of $p$).
	Moreover, let $a,b \in \bR$ with $a<b$ or $a \in \bR, b=+\infty$.
	Then, for $\bfk=\bF_2=\bZ/2\bZ$, one has
	\begin{equation}
	\begin{split}
	& \sum_{a \le f_{21}(C_j) <b} \; \sum_{k \in \bZ} \dim_{\bF_2} H^k(C_j;\bF_2) \\
	& \qquad \ge
	\sum_{k \in \bZ} \dim_{\bF_2}
	H^k \RG_{M \times [a,b)}((-\infty,b);\cHom^\star(F_2,F_1)).
	\end{split}
	\end{equation}
	In particular,
	\begin{equation}
	\sum_{j=1}^n \sum_{k \in \bZ} \dim_{\bF_2} H^k(C_j;\bF_2) \ge
	\sum_{k \in \bZ} \dim_{\bF_2} \Hom_{\cT(M)}(F_2,F_1[k]).
	\end{equation}
	If $L_1$ and $L_2$ intersect transversally, the inequalities hold for any field $\bfk$, not only for $\bF_2$.
\end{theorem}

We also have 
\begin{equation}\label{eq:introisomcoh}
\Hom_{\cT(M)}(F_2,F_1[k])
\simeq
H^k(M;\cL)
\quad \text{for any $k \in \bZ$},
\end{equation}
where $\cL$ is the locally constant sheaf of rank $1$ on $M$ associated with $F_1$ and $F_2$  (see \pref{prp:cohlagr} for details).
Combining this with \pref{thm:intromain}, we obtain a purely sheaf-theoretic proof of the following result of Nadler~\cite{Nad} and Fukaya--Seidel--Smith~\cite{FSS08}, as a corollary.

\begin{corollary}[{\cite[Thm.~1.3.1]{Nad} and \cite[Thm.~1]{FSS08}}]
	Let $L_1$ and $L_2$ be compact connected exact Lagrangian submanifolds of $T^*M$ intersecting transversally.
	Then
	\begin{equation}
	\#(L_1 \cap L_2) \ge \sum_{k \in \bZ} \dim H^k(M;\cL)
	\end{equation}
	for any rank $1$ locally constant sheaf $\cL$ on $M$ over any field $\bfk$.
	In particular, $\#(L_1 \cap L_2) \ge \sum_{k \in \bZ} \dim H^k(M;\bfk)$.
\end{corollary}

The proof of \pref{thm:intromain} goes as follows.
First, we apply the Morse--Bott inequality for sheaves (see \pref{thm:morsebottineq}) to $H\coloneqq \cHom^\star(F_2,F_1)$ and the function $M \times \bR \to \bR, (x,t) \mapsto t$, and obtain
\begin{equation}\label{eq:introineqloccoh}
\begin{split}
& \sum_{a \le c<b} \sum_{k \in \bZ} \dim H^k \RG \left( M \times \{c\};\RG_{M \times [c,+\infty)}(H)|_{M \times \{ c \}} \right)  \\
& \qquad \ge
\sum_{k \in \bZ} \dim H^k \RG_{M \times [a,b)}(M \times (-\infty,b);H).
\end{split}
\end{equation}
In order to calculate the left-hand side of \eqref{eq:introineqloccoh}, we use the functor $\mu hom \colon \Db(X)^{\mathrm{op}} \times \Db(X) \to \Db(T^*X)$ introduced by Kashiwara--Schapira~\cite{KS90}.
Using the functor, we show the isomorphism
\begin{equation}\label{eq:introisommuhom}
\RG \left( M \times \{c\};\RG_{M \times [c,+\infty)}(H)|_{M \times \{ c \}} \right)
\simeq
\RG(\Omega_+;\mu hom({T_c}_*F_2,F_1)|_{\Omega_+}),
\end{equation}
where $T_c \colon M \times \bR \to M \times \bR, (x,t) \mapsto (x,t+c)$ and $\Omega_+\coloneqq \{\tau >0 \} \subset T^*(M \times \bR)$ with $(t;\tau)$ being the homogeneous symplectic coordinate on $T^*\bR$.
The object $\mu hom({T_c}_*F_2,F_1)|_{\Omega_+}$ is supported in $\{ (x,t;\tau \xi,\tau) \mid \tau >0, (x;\xi) \in L_1 \cap L_2, t=f_2(x;\xi)-f_1(x;\xi)=c \}$ and isomorphic to a shift of the constant sheaf of rank $1$ on the support.
This completes the proof.

\begin{remark}
	Even if the intersection is degenerate, \eqref{eq:introineqloccoh} and \eqref{eq:introisommuhom} still hold, but the object $\mu hom({T_c}_*F_2,F_1)|_{\Omega_+}$ is not necessarily locally constant on the support.
	In this sense, the family of sheaves $\{\mu hom({T_c}_*F_2,F_1)|_{\Omega_+}\}_c$ encodes the ``contribution" from each possibly degenerate component of the intersection $L_1 \cap L_2$.
	We will also explore the contribution in degenerate cases in \pref{sec:nonclean}.
\end{remark}

\subsection{Relation to Lagrangian intersection Floer theory}

Although our approach is purely sheaf-theoretic, it seems to be closely related to Floer cohomology and Fukaya categories.
We briefly remark on the relation below.
Tamarkin's category $\cT(M)$ has the following properties:
\begin{enumerate}
	\item Hamiltonian invariance (\cite{Tamarkin, GS14}),
	\item the dimension of the cohomology of the clean intersection of two compact exact Lagrangian submanifolds is bounded from below by the dimension of the Hom space of simple sheaf quantizations (\pref{thm:intromain}).
\end{enumerate}
Moreover, as pointed out by T.~Kuwagaki, the following also holds in $\cT(M)$:
\begin{itemize}
	\item[(iii)] a simple sheaf quantization associated with any compact connected exact Lagrangian submanifold is isomorphic to a simple sheaf quantization associated with the zero-section of $T^*M$ (see \pref{prp:isomT}).
\end{itemize}

The Floer cohomology $\mathrm{HF}^*(L_2,L_1)$ has similar properties to (i) and (ii), though the approach is totally different.
Floer cohomology for clean Lagrangian intersections was studied by Po{\'z}niak~\cite{Poz99}, Frauenfelder~\cite{Frauenfelder}, Fukaya--Oh--Ohta--Ono~\cite{FOOO09,FOOO092}, and Schm{\"a}schke~\cite{schmaschke2016floer}.
Moreover, Nadler~\cite{Nad} and Fukaya--Seidel--Smith~\cite{FSS08,FSS09} proved the following, which corresponds to (iii): in the infinitesimal Fukaya category of $T^*M$, any relative spin compact connected exact Lagrangian submanifold of $T^*M$ with vanishing Maslov class is isomorphic to a shift of the zero-section.
Note that their assumptions of relative spin and vanishing Maslov class can be removed, thanks to results of Abouzaid~\cite{Abouzaid}, and Abouzaid--Kragh~\cite{Kragh}, respectively.
We also remark that Guillermou~\cite{Gu12,Gulec} gave a sheaf-theoretic proof for the relative spin property and the vanishing of the Maslov class.

During the preparation of this paper, C.~Viterbo announced\footnote{In a seminar at IMJ-PRG on 10 October 2016.} that he had found some relation between $\cHom^\star(F_2,F_1)$ and the Floer cochain complex $\mathrm{CF}(L_2,L_1)$.

\subsection{Outline of this paper}

This paper is organized as follows.
In \pref{sec:preliminaries}, we recall the microlocal sheaf theory due to Kashiwara and Schapira~\cite{KS90}.
In \pref{sec:quantandnondisp} we review the results of \cite{Tamarkin,GKS,GS14,Gu12,Gulec} on Tamarkin's non-displaceability theorem, and sheaf quantization of Hamiltonian isotopies and compact exact Lagrangian submanifolds in cotangent bundles.
In \pref{sec:intersection} we prove the isomorphism~\eqref{eq:introisomcoh} and \pref{thm:intromain}.
In \pref{sec:nonclean} we briefly remark that our method can deal with degenerate Lagrangian intersections, using very simple examples.
In \pref{sec:appfunctorial} we prove the ``functoriality" of simple sheaf quantizations with respect to Hamiltonian diffeomorphisms.
In \pref{sec:appasano} by Tomohiro~Asano, we relate the shift of a simple sheaf quantization of a Lagrangian to the grading in Lagrangian intersection Floer cohomology theory.

\subsection*{Acknowledgments}
\addcontentsline{toc}{subsection}{Acknowledgments}

The author wishes to express his sincere gratitude to St{\'e}phane~Guillermou for many helpful discussions.
He is also very grateful to Pierre~Schapira for many enlightening discussions and helpful advice and to Tomohiro~Asano for many fruitful discussions and kindly writing \pref{sec:appasano}.
The author also thanks Vincent~Humili{\`e}re and Tatsuki~Kuwagaki for many stimulating discussions.
He expresses his gratitude to IMJ-PRG and ``equipe Analyse Alg{\'e}brique" for hospitality during the preparation of this paper.
This work was supported by a Grant-in-Aid for JSPS Fellows 15J07993 and the Program for Leading Graduate Schools, MEXT, Japan.

\section{Preliminaries on microlocal sheaf theory}\label{sec:preliminaries}

In this paper, all manifolds are assumed to be real manifolds of class $C^\infty$ without boundary.
Throughout this paper, let $\bfk$ be a field.

In this section we recall some definitions and results from \cite{KS90}.
We mainly follow the notation in \cite{KS90}.
Until the end of this section, let $X$ be a $C^\infty$-manifold without boundary.

\subsection{Geometric notions (\cite[\S4.3, \S A.2]{KS90})}\label{subsec:geometric}

For a locally closed subset $A$ of $X$, we denote by $\overline{A}$ its closure and by $\Int(A)$ its interior.
We also denote by $\Delta_X$ or simply $\Delta$ the diagonal of $X \times X$.
We denote by $\tau_X \colon TX \to X$ the tangent bundle of $X$, and by $\pi_X \colon T^*X \to X$ the cotangent bundle of $X$.
If there is no risk of confusion, we simply write $\tau$ and $\pi$ instead of $\tau_X$ and $\pi_X$, respectively.
For a submanifold $M$ of $X$, one denotes by $T_MX$ the normal bundle to $M$ in $X$, and by $T^*_MX$ the conormal bundle to $M$ in $X$.
In particular, $T^*_XX$ denotes the zero-section of $T^*X$.
We set $\rT X\coloneqq T^*X \setminus T^*_XX$.
For two subsets $S_1$ and $S_2$ of $X$, we denote by $C(S_1,S_2) \subset TX$ the normal cone of the pair $(S_1,S_2)$.

Let $f \colon X \to Y$ be a morphism of manifolds.
With $f$ we associate the following morphisms and commutative diagram:
\begin{equation}\label{diag:fpifd}
\begin{split}
\xymatrix{
	T^*X \ar[d]_-{\pi_X} & X \times_Y T^*Y \ar[d]^-\pi \ar[l]_-{f_d} \ar[r]^-{f_\pi} & T^*Y \ar[d]^-{\pi_Y} \\
	X \ar@{=}[r] & X \ar[r]_-f & Y,
}
\end{split}
\end{equation}
where $f_\pi$ is the projection and $f_d$ is induced by the transpose of the tangent map $f' \colon TX \to X \times_Y TY$.

We denote by $(x;\xi)$ a local homogeneous coordinate system on $T^*X$.
The cotangent bundle $T^*X$ is an exact symplectic manifold with the Liouville 1-form $\alpha=\langle \xi, dx \rangle$.
We denote by $a \colon T^*X \to T^*X,(x;\xi) \mapsto (x;-\xi)$ the antipodal map.
For a subset $A$ of $T^*X$, we denote by $A^a$ its image under the map $a$.
We also denote by $\bfh \colon T^*T^*X \simto TT^*X$ the Hamiltonian isomorphism given in local coordinates by $\bfh(dx_i)=-\partial / \partial\xi_i$ and $\bfh(d\xi_i)=\partial / \partial x_i$.

\subsection{Microsupports of sheaves (\cite[\S5.1, \S5.4, \S6.1]{KS90})}\label{subsec:microsupport}

We denote by $\bfk_X$ the constant sheaf with stalk $\bfk$ and by $\Module(\bfk_X)$ the abelian category of sheaves of $\bfk$-vector spaces on $X$.
Moreover, we denote by $\Db(X)=\Db(\Module(\bfk_X))$ the bounded derived category of $\Module(\bfk_X)$.
One can define Grothendieck's six operations between derived categories of sheaves $\cRHom,\allowbreak \otimes, \allowbreak Rf_*,\allowbreak f^{-1},\allowbreak Rf_!,\allowbreak f^!$ for a morphism of manifolds $f \colon X \to Y$.
Since we work over the field $\bfk$, we simply write $\otimes$ instead of $\lten$.
Moreover, for $F \in \Db(X)$ and $G \in \Db(Y)$, we define their external tensor product $F \boxtimes G \in \Db(X \times Y) $ by $F \boxtimes G \coloneqq q_X^{-1}F \otimes q_Y^{-1}G$, where $q_X \colon X \times Y \to X$ and $q_Y \colon X \times Y \to Y$ are the projections.
For a locally closed subset $Z$ of $X$, we denote by $\bfk_Z$ the zero-extension of the constant sheaf with stalk $\bfk$ on $Z$ to $X$, extended by $0$ on $X \setminus Z$.
Moreover, for a locally closed subset $Z$ of $X$ and $F \in \Db(X)$, we define $F_Z, \RG_Z(F) \in \Db(X)$ by
\begin{equation}
F_Z\coloneqq F \otimes \bfk_Z, \quad \RG_Z(F)\coloneqq \cRHom(\bfk_Z,F).
\end{equation}
One denotes by $\omega_X \in \Db(X)$ the dualizing complex on $X$, that is, $\omega_X\coloneqq a_X^!\bfk$, where $a_X \colon X \to \pt$ is the natural morphism.
Note that $\omega_X$ is isomorphic to $\ori_X[\dim X]$, where $\ori_X$ is the orientation sheaf on $X$.
More generally, for a morphism of manifolds $f \colon X \to Y$, we denote by $\omega_f=\omega_{X/Y}\coloneqq f^!\bfk_Y \simeq \omega_X \otimes f^{-1}\omega_Y^{\otimes -1}$ the relative dualizing complex.

Let us recall the definition of the \emph{microsupport} $\MS(F)$ of $F \in \Db(X)$.

\begin{definition}[{\cite[Def.~5.1.2]{KS90}}]\label{def:microsupport}
	Let $F \in \Db(X)$ and $p \in T^*X$.
	One says that $p \not\in \MS(F)$ if there is a neighborhood $U$ of $p$ in $T^*X$ such that for any $x_0 \in X$ and any $C^\infty$-function $\varphi$ on $X$ (defined on a neighborhood of $x_0$) satisfying $d\varphi(x_0) \in U$, one has $\RG_{\{\varphi \ge \varphi(x_0)\}}(F)_{x_0} \simeq 0$.
\end{definition}

One can check the following properties:
\begin{enumerate}
	\item The microsupport of an object in $\Db(X)$ is a conic (i.e., invariant under the action of $\bR_{>0}$ on $T^*X$) closed subset of $T^*X$.
	\item For an object $F \in \Db(X)$, one has $\MS(F) \cap T^*_XX=\pi(\MS(F))=\Supp(F)$.
	\item The microsupports satisfy the triangle inequality: if $F_1 \to F_2 \to F_3 \overset{+1}{\to}$ is a distinguished triangle in $\Db(X)$, $\MS(F_i) \subset \MS(F_j) \cup \MS(F_k)$ for $j \neq k$.
\end{enumerate}
We also use the notation $\mathring{\MS}(F)\coloneqq \MS(F) \cap \rT X=\MS(F) \setminus T^*_XX$.

We denote by $\BD(X)=\BD(\Module(\bfk_X))$ the (unbounded) derived category of sheaves of $\bfk$-vector spaces on $X$.
An object $F \in \BD(X)$ is said to be locally bounded if for any relatively compact open subset $U$ of $X$, one has $F|_U \in \Db(U)$.
We denote by $\Dlb(X)$ the full subcategory of $\BD(X)$ consisting of locally bounded objects.
The microsupport of an object in $\Dlb(X)$ can be defined in totally the same way as in \pref{def:microsupport}, since it is a local notion.

\begin{example}
	(i) If $F$ is a non-zero locally constant sheaf on a connected manifold $X$, then $\MS(F)=T^*_XX$.
	Conversely, if $\MS(F) \subset T^*_XX$ then the cohomology sheaves $H^k(F)$ are locally constant for all $k \in \bZ$.
	\smallskip
	
	\noindent (ii) Let $M$ be a closed submanifold of $X$.
	Then $\MS(\bfk_M)=T^*_MX \subset T^*X$.
	\smallskip
	
	\noindent (iii) Let $\varphi \colon X \to \bR$ be a $C^\infty$-function and assume that $d\varphi(x) \neq 0$ for any $x \in \varphi^{-1}(0)$.
	Set $U\coloneqq \{x \in X \mid \varphi(x)>0\}$ and $Z\coloneqq \{x \in X\mid \varphi(x) \ge 0\}$.
	Then
	\begin{equation}
	\begin{split}
	\MS(\bfk_U)=T^*_XX|_U \cup
	\{(x;\lambda d\varphi(x)) \mid \varphi(x)=0, \lambda \le 0\}, \\
	\MS(\bfk_Z)=T^*_XX|_Z \cup
	\{(x;\lambda d\varphi(x)) \mid \varphi(x)=0, \lambda \ge 0\}.
	\end{split}
	\end{equation}
\end{example}

The following proposition is called (a particular case of) the microlocal Morse lemma.
See \cite[Prop.~5.4.17 and Cor.~5.4.19]{KS90} for more details.
The classical theory corresponds to the case $F$ is the constant sheaf $\bfk_X$.

\begin{proposition}\label{prp:microlocalmorse}
	Let $F \in \Db(X)$ and $\varphi \colon X \to \bR$ be a $C^\infty$-function.
	Moreover, let $a,b \in \bR$ with $a<b$ or $a \in \bR, b=+\infty$.
	Assume that 
	\begin{enumerate}
		\renewcommand{\labelenumi}{$\mathrm{(\arabic{enumi})}$}
		\item $\varphi$ is proper on $\Supp(F)$,
		\item  $d\varphi(x) \not\in \MS(F)$ for any $x \in \varphi^{-1}([a,b))$.
	\end{enumerate}
	Then the canonical morphism
	\begin{equation}
	\RG(\varphi^{-1}((-\infty,b));F)
	\lto
	\RG(\varphi^{-1}((-\infty,a));F)
	\end{equation}
	is an isomorphism.
\end{proposition}

By using microsupports, we can microlocalize the category $\Db(X)$.
Let $A \subset T^*X$ be a subset and set $\Omega=T^*X \setminus A$.
We denote by $\Db_{A}(X)$ the subcategory of $\Db(X)$ consisting of sheaves whose microsupports are contained in $A$.
By the triangle inequality, the subcategory $\Db_{A}(X)$ is a triangulated subcategory.
We define $\Db(X;\Omega)$ as the localization of $\Db(X)$ by $\Db_A(X)$: $\Db(X;\Omega)\coloneqq \Db(X)/\Db_A(X)$.
A morphism $u \colon F \to G$ in $\Db(X)$ becomes an isomorphism in $\Db(X;\Omega)$ if $u$ is embedded in a distinguished triangle $F \overset{u}{\to} G \to H \overset{+1}{\to}$ with $\MS(H) \cap \Omega=\emptyset$.
For a closed subset $B$ of $\Omega$, $\Db_B(X;\Omega)$ denotes the full triangulated subcategory of $\Db(X;\Omega)$ consisting of $F$ with $\MS(F) \cap \Omega \subset B$.
In the case $\Omega=\{p\}$ with $p \in T^*X$, we simply write $\Db(X;p)$ instead of $\Db(X;\{p\})$.
Note that our notation is the same as in \cite{KS90} and slightly differs from that of \cite{Gu12,Gulec}.

\subsection{Functorial operations (\cite[\S5.4]{KS90})}

We consider bounds for the microsupports of proper direct images, non-characteristic inverse images, and $\cRHom$.

\begin{definition}[{\cite[Def.~5.4.12]{KS90}}]
	Let $f \colon X \to Y$ be a morphism of manifolds and $A$ be a closed conic subset of $T^*Y$.
	The morphism $f$ is said to be \emph{non-characteristic} for $A$ if
	\begin{equation}
	f_\pi^{-1}(A) \cap f_d^{-1}(T^*_XX) \subset X \times_Y T^*_YY.
	\end{equation}
\end{definition}

See \eqref{diag:fpifd} for the notation $f_\pi$ and $f_d$.
In particular, any submersion from $X$ to $Y$ is non-characteristic for any closed conic subset of $T^*Y$.
Note that submersions are called smooth morphisms in \cite{KS90}.
One can show that if $f \colon X \to Y$ is non-characteristic for a closed conic subset $A$ of $T^*Y$, then $f_d f_\pi^{-1}(A)$ is a closed conic subset of $T^*X$.

\begin{theorem}[{\cite[Prop.~5.4.4 and Prop.~5.4.13]{KS90}}]\label{thm:operations}
	Let $f \colon X \to Y$ be a morphism of manifolds, $F \in \Db(X)$, and $G \in \Db(Y)$.
	\begin{enumerate}
		\item Assume that $f$ is proper on $\Supp(F)$.
		Then $\MS(Rf_*F) \subset f_\pi f_d^{-1}(\MS(F))$.
		\item Assume that $f$ is non-characteristic for $\MS(G)$.
		Then the canonical morphism \linebreak $f^{-1}G \otimes \omega_{f} \to f^!G$ is an isomorphism and $\MS(f^{-1}G) \cup \MS(f^!G) \subset f_d f_\pi^{-1}(\MS(G))$.
	\end{enumerate}
\end{theorem}

\begin{proposition}[{\cite[Prop.~5.4.2]{KS90}}]\label{prp:SSprod}
	For $i=1,2$, let $X_i$ be a manifold and denote by $q_i$ the projection $X_1 \times X_2 \to X_i$.
	Moreover, let $F_i \in \Db(X_i)$ for $i=1,2$.
	Then 
	\begin{equation}
	\MS(\cRHom(q_2^{-1}F_2,q_1^{-1}F_1))
	\subset
	\MS(F_1) \times \MS(F_2)^a.
	\end{equation}
\end{proposition}

Using \pref{prp:SSprod} and \pref{thm:operations}(ii) for the diagonal embedding $\delta \colon X \to X \times X$, one can prove the following:

\begin{proposition}[{\cite[Prop.~5.4.14~(ii)]{KS90}}]\label{prp:SSHom}
	Let $F, G \in \Db(X)$ and assume that $\MS(F) \cap \MS(G) \subset T^*_XX$.
	Then
	\begin{equation}
	\MS(\cRHom(F,G)) \subset \MS(F)^a+\MS(G),
	\end{equation}
	where $+$ is the fiberwise sum.
\end{proposition}

\subsection{Non-proper direct images (\cite{Tamarkin, GS14})}\label{subsec:nonproper}

We consider estimates of the microsupports of non-proper direct images in special cases.
Let $V_1$ and $V_2$ be finite-dimensional real vector spaces and consider a constant linear map $u \colon X \times V_1 \to X \times V_2$.
That is, we assume that there exists a linear map $u_V \colon V_1 \to V_2$ satisfying $u=\id_X \times u_V$.
The map $u$ induces the maps
\begin{equation}\label{diag:nonproper}
\begin{split}
\xymatrix{
	& T^*X \times V_1 \times V_2^* \ar[ld]_-{u_d} \ar[rd]^-{u_\pi} & \\
	T^*X \times V_1 \times V_1^* \ar[rd]_-{v_\pi} & & T^*X \times V_2 \times V_2^* \ar[ld]^-{v_d} \\
	& T^*X \times V_2 \times V_1^*. &
}
\end{split}
\end{equation}
Note that for a subset $A$ of $T^*(X \times V_1)$, we have $u_\pi(u_d^{-1}(A))=v_d^{-1}(v_\pi(A))$.

\begin{definition}
	Let $u \colon X \times V_1 \to X \times V_2$ be a constant linear map and $A \subset T^*(X \times V_1)$ be a closed subset.
	One sets
	\begin{equation}
	u_\sharp(A)\coloneqq v_d^{-1}\left(\overline{v_\pi(A)}\right).
	\end{equation}
\end{definition}

\begin{proposition}[{\cite[Lem.~3.3]{Tamarkin} and \cite[Thm.~1.16]{GS14}}]\label{prp:nonproper}
	Let $u \colon X \times V_1 \to X \times V_2$ be a constant linear map and $F \in \Db(X \times V_1)$.
	Then
	\begin{equation}
	\MS(Ru_*F) \cup \MS(Ru_!F)
	\subset u_\sharp(\MS(F)).
	\end{equation}
\end{proposition}

\subsection{Morse--Bott inequality for sheaves (\cite{ST92})}\label{subsec:morse-bott-inequality-for-sheaves}

In this subsection, we give the Morse--Bott inequality for sheaves, which is a slight generalization of the Morse inequality for sheaves by Kashiwara--Schapira~\cite[Prop.~5.4.20]{KS90} and was proved by Schapira--Tose~\cite{ST92}.
For a bounded complex $W$ of $\bfk$-vector spaces with finite-dimensional cohomology and $k \in \bZ$, we set
\begin{equation}
b_k(W)\coloneqq \dim H^k(W).
\end{equation}
Let $F \in \Db(X)$ and $\varphi \colon X \to \bR$ be a $C^\infty$-function.
We set
\begin{equation}
\Gamma_{d\varphi}
\coloneqq 
\{(x;d\varphi(x)) \mid x \in X\} \subset T^*X.
\end{equation}
We consider the following assumptions:
\begin{enumerate}
	\renewcommand{\labelenumi}{$\mathrm{(\arabic{enumi})}$}
	\item $\Supp(F) \cap \varphi^{-1}((-\infty,t])$ is compact for any $t \in \bR$,
	\item the set $\varphi(\pi(\MS(F) \cap \Gamma_{d\varphi}))$ is finite, say
	$\{c_1,\dots,c_N\}$ with $c_1<\dots<c_N$,
	\item the object
	\begin{equation}
	W_i\coloneqq \RG\left( \varphi^{-1}(c_i);\RG_{\{\varphi \ge c_i\}}(F)|_{\varphi^{-1}(c_i)} \right)
	\end{equation}
	has finite-dimensional cohomology for any $i=1,\dots,N$.
\end{enumerate}

\begin{theorem}[{\cite[Thm.~1.1]{ST92}, see also \cite[Prop.~5.4.20]{KS90}}]\label{thm:morsebottineq}
	Assume that \emph{(1)--(3)} are satisfied.
	Then
	\begin{enumerate}
		\item $\RG(X;F)$ has finite-dimensional cohomology,
		\item one has
		\begin{equation}\label{eq:morseineq}
			b_k(\RG(X;F)) \le \sum_{i=1}^N b_k(W_i)
		\end{equation}
		for any $k \in \bZ$.
	\end{enumerate}
\end{theorem}

Note that \cite[Thm.~1.1]{ST92} is a stronger result than \pref{thm:morsebottineq}.
In this paper, we only use the weaker inequality~\eqref{eq:morseineq}.
The proof is the same as \cite[Prop.~5.4.20]{KS90}, since
\begin{equation}
\RG_{[t,+\infty)}(R\varphi_* F)_t
\simeq
\RG\left( \varphi^{-1}(t);\RG_{\{\varphi \ge t \}}(F)|_{\varphi^{-1}(t)} \right).
\end{equation}

\subsection{Kernels (\cite[\S3.6]{KS90})}\label{subsec:kernels}

For $i=1,2,3$, let $X_i$ be a manifold.
We write $X_{ij}\coloneqq X_i \times X_j$ and $X_{123}\coloneqq X_1 \times X_2 \times X_3$ for short.
We use the same symbol $q_i$ for the projections $X_{ij} \to X_i$ and $X_{123} \to X_i$.
We also denote by $q_{ij}$ the projection $X_{123} \to X_{ij}$.
Similarly, we denote by $p_{ij}$ the projection $T^*X_{123} \to T^*X_{ij}$.
One denotes by $p_{12^a}$ the composite of $p_{12}$ and the antipodal map on $T^*X_2$.

Let $A \subset T^*X_{12}$ and $B \subset T^*X_{23}$.
We set
\begin{equation}\label{eq:compset}
A \circ B
\coloneqq 
p_{13}(p_{12^a}^{-1}A \cap p_{23}^{-1}B) \subset T^*X_{13}.
\end{equation}
We define the composition of kernels as follows:
\begin{equation}
\begin{split}
\underset{X_2}{\circ} \colon \Db(X_{12}) \times \Db(X_{23}) & \to \Db(X_{13}) \\
(K_{12},K_{23}) & \mapsto K_{12} \underset{X_2}{\circ} K_{23}
\coloneqq 
R {q_{13}}_!\,(q_{12}^{-1}K_{12}\otimes q_{23}^{-1}K_{23}).
\end{split}
\end{equation}
If there is no risk of confusion, we simply write $\circ$ instead of $\underset{X_2}{\circ}$.
By \pref{thm:operations} and estimates of the microsupports of tensor products (see \cite[Prop.~5.4.14]{KS90}), we have the following proposition.

\begin{proposition}\label{prp:SScomp}
	Let $K_{ij} \in \Db(X_{ij})$ and set $\Lambda_{ij}\coloneqq \MS(K_{ij}) \subset T^*X_{ij} \ (ij=12,23)$.
	Assume 
	\begin{enumerate}
		\renewcommand{\labelenumi}{$\mathrm{(\arabic{enumi})}$}
		\item $q_{13}$ is proper on $q_{12}^{-1}\Supp(K_{12}) \cap q_{23}^{-1}\Supp(K_{23})$,
		\item $p_{12^a}^{-1}\Lambda_{12} \cap p_{23}^{-1}\Lambda_{23} \cap (T^*_{X_1}X_1 \times T^*X_2 \times T^*_{X_3}X_3) \subset T^*_{X_{123}}X_{123}$.
	\end{enumerate}
	Then
	\begin{equation}
	\MS(K_{12} \underset{X_2}{\circ} K_{23}) \subset
	\Lambda_{12} \circ \Lambda_{23}.
	\end{equation}
\end{proposition}

\subsection{Microlocalization and $\mu hom$ functors (\cite[\S4.3, \S4.4]{KS90})}

Let $M$ be a closed submanifold of $X$.
The microlocalization functor along $M$ is a functor $\mu_M \colon \Db(X) \to \Db(T^*_MX)$ (see \cite[\S4.3]{KS90} for more details).
Microlocalization is related to local cohomology as follows.
Let $p \in \rT X$ and $\varphi \colon X \to \bR$ be a $C^\infty$-function such that $\varphi(\pi(p))=0$ and $d \varphi(\pi(p))=p$.
Then, for $F \in \Db(X)$, we have 
\begin{equation}\label{eq:microlocalizationstalk}
\RG_{\{\varphi \ge 0 \}}(F)_{\pi(p)}
\simeq
\mu_{\varphi^{-1}(0)}(F)_{p}.
\end{equation}

Under suitable assumptions, the functoriality of microlocalization with respect to proper direct images and non-characteristic inverse images holds as follows:

\begin{proposition}[{\cite[Prop.~4.3.4 and Cor.~6.7.3]{KS90}}]\label{prp:mufunctorial}
	Let $f \colon X \to Y$ be a morphism of manifolds.
	Moreover, let $N$ be a closed submanifold of $Y$ and assume that $M=f^{-1}(N)$ is also a closed submanifold of $X$.
	Denote by $f_{Md} \colon M \times_N T^*_NY \to T^*_MX$ the morphism induced by $f_d$ and by $f_{M\pi} \colon M \times_N T^*_NY \to T^*_NY$ the morphism induced by $f_\pi$ (see \eqref{diag:fpifd}).
	\begin{enumerate}
		\item Let $F \in \Db(X)$.
		Assume that $f$ is proper on $\Supp(F)$ and $f_{Md} \colon M \times_N T^*_NY \to T^*_MX$ is surjective. 
		Then 
		\begin{equation}
		R{f_{M\pi}}_! f_{Md}^{-1} \, \mu_M(F) \simto \mu_N(Rf_*F).
		\end{equation}
		\item Let $G \in \Db(Y)$.
		Assume that $f$ is non-characteristic for $\MS(G)$ and $f|_M \colon M \to N$ is a submersion. 
		Then 
		\begin{equation}
		\mu_M(f^!G) \simto R {f_{Md}}_* f_{M\pi}^! \, \mu_N(G).
		\end{equation}
	\end{enumerate}	
\end{proposition}

We also recall the functor $\mu hom$.
Let $q_1, q_2 \colon X \times X \to X$ be the projections.
We identify $T^*_{\Delta_X}(X \times X)$ with $T^*X$ through the first projection $(x,x;\xi,-\xi) \mapsto (x;\xi)$.

\begin{definition}[{\cite[Def.~4.4.1]{KS90}}]
	For $F, G \in \Db(X)$, one defines
	\begin{equation}
	\mu hom (F,G)
	\coloneqq 
	\mu_{\Delta_X}\cRHom(q_2^{-1}F,q_1^!G) 
	\in	\Db(T^*X).
	\end{equation}
\end{definition}

\begin{proposition}[{\cite[Prop.~4.4.2 and Prop.~4.4.3]{KS90}}]\label{prp:muhomprop}
	Let $F,G \in \Db(X)$.
	\begin{enumerate}
		\item $R \pi_* \mu hom(F,G) \simeq \cRHom(F,G)$.
		\item If $F$ is cohomologically constructible (see \emph{\cite[\S3.4]{KS90}} for the definition), then \linebreak $R \pi_! \mu hom(F,G) \simeq \cRHom(F,\bfk_X) \otimes G$.
		\item For a closed submanifold $M$ of $X$, $\mu hom(\bfk_{M},F)	\simeq i_* \mu_M(F)$, where $i \colon T^*_M X \to T^*X$ is the embedding.
	\end{enumerate}
\end{proposition}

\begin{proposition}[{\cite[Cor.~5.4.10 and Cor.~6.4.3]{KS90}}]\label{prp:muhomss}
	Let  $F,G \in \Db(X)$.
	Then
	\begin{equation}
	\begin{split}
	\Supp(\mu hom(F,G)) & \subset \MS(F) \cap \MS(G), \\
	\MS(\mu hom(F,G)) & \subset -\bfh^{-1}(C(\MS(G), \MS(F))),
	\end{split}
	\end{equation}
	where $C(S_1,S_2)$ is the normal cone and $\bfh \colon T^*T^*X \simto T T^*X$ is the Hamiltonian isomorphism (see \pref{subsec:geometric}).
\end{proposition}

\begin{proposition}\label{prp:microlocalcoh}
	Let $\varphi \colon X \to \bR$ be a $C^\infty$-function and assume that $d\varphi(x) \neq 0$ for any $x \in \varphi^{-1}(0)$.
	Set $M \coloneqq  \varphi^{-1}(0)$ and define an open subset $T^{*+}_MX$ of $T^*_{M}X$ by
	\begin{equation}
	T^{*+}_MX
	\coloneqq 
	\{ (x;\lambda d\varphi(x)) \mid x \in M, \lambda >0 \}.
	\end{equation}
	Moreover, denote by $\pi_{M+} \colon T^{*+}_MX \to M$ the projection.
	Let $F \in \Db(X)$.
	Then
	\begin{equation}\label{eq:muhalf}
	\RG_{\{\varphi \ge 0\}}(F)|_{M}
	\simeq
	R {\pi_{M+}}_* \mu hom(\bfk_{\{ \varphi \ge 0 \}},F)|_{T^{*+}_MX}
	\simeq 
	R {\pi_{M+}}_* \mu_{M}(F)|_{T^{*+}_MX}.
	\end{equation}
	In particular, 	
	\begin{equation}
	\RG(M;\RG_{\{\varphi \ge 0\}}(F)|_{M})
	\simeq
	\RG\left( T^{*+}_MX;\mu_{M}(F)|_{T^{*+}_MX} \right).
	\end{equation}
\end{proposition}

\begin{proof}
	Consider the distinguished triangle
	\begin{equation}\label{eq:muhomtri}
	R\pi_! \mu hom(\bfk_{\{ \varphi \ge 0 \}},F)
	\to 
	R \pi_* \mu hom(\bfk_{\{ \varphi \ge 0 \}},F)
	\to 
	R \mathring{\pi}_* \mu hom(\bfk_{\{ \varphi \ge 0 \}},F)|_{\rT X}
	\overset{+1}{\to}.
	\end{equation}
	By \pref{prp:muhomss}, $\Supp(\mu hom(\bfk_{\{ \varphi \ge 0 \}},F)|_{\rT X}) \subset T^{*+}_MX$.
	Hence we have
	\begin{equation}
	R \mathring{\pi}_* \mu hom(\bfk_{\{ \varphi \ge 0 \}},F)|_{\rT X} 
	\simeq 
	\left( R {\pi_{M+}}_* \mu hom(\bfk_{\{ \varphi \ge 0 \}},F)|_{T^{*+}_MX} \right)_M.
	\end{equation} 
	On the other hand, since $\bfk_{\{ \varphi \ge 0 \}}$ is cohomologically constructible, by \pref{prp:muhomprop}(i) and (ii), we get 
	\begin{equation}
	\begin{split}
	& R\pi_! \mu hom(\bfk_{\{ \varphi \ge 0 \}},F)
	\simeq \cRHom(\bfk_{\{ \varphi \ge 0 \}},\bfk_X) \otimes F
	\simeq \RG_{\{\varphi \ge 0 \}}(\bfk_X) \otimes F, \\
	& R\pi_* \mu hom(\bfk_{\{ \varphi \ge 0 \}},F)
	\simeq \cRHom(\bfk_{\{ \varphi \ge 0 \}},F)
	\simeq \RG_{\{\varphi \ge 0 \}}(F).
	\end{split}
	\end{equation}
	Since $\RG_{\{\varphi \ge 0 \}}(\bfk_X)|_M \simeq 0$, restricting the distinguished triangle \eqref{eq:muhomtri} to $M$, we obtain the first isomorphism in \eqref{eq:muhalf}.
	Moreover since $\MS(\bfk_{\{ \varphi>0 \}}) \cap {T^{*+}_MX}=\emptyset$, by \pref{prp:muhomss}, we have 
	\begin{equation}
	\mu hom(\bfk_{\{ \varphi \ge 0 \}},F)|_{T^{*+}_MX}
	\simto 
	\mu hom(\bfk_{\{ \varphi=0 \}},F)|_{T^{*+}_MX}.
	\end{equation}
	Thus the second isomorphism in \eqref{eq:muhalf} follows from \pref{prp:muhomprop}(iii).	
\end{proof}

\subsection{Simple sheaves and quantized contact transformations (\cite[\S7.5]{KS90})}\label{subsec:simple}

Let $\Lambda \subset \rT X$ be a locally closed conic Lagrangian submanifold and $p \in \Lambda$.
Simple sheaves along $\Lambda$ at $p$ are defined in \cite[Def.~7.5.4]{KS90}.
In this subsection we recall them.

Let $\varphi \colon X \to \bR$ be a $C^\infty$-function such that $\varphi(\pi(p))=0$ and $\Gamma_{d\varphi}$ intersects $\Lambda$ transversally at $p$.
For $p \in \Gamma_{d\varphi} \cap \Lambda$, we define the following Lagrangian subspaces in $T_pT^*X$:
\begin{equation}\label{eq:notaionlambda}
\lambda_\infty(p)\coloneqq T_p (T^*_{\pi(p)}X), \quad
\lambda_{\Lambda}(p)\coloneqq T_p \Lambda, \quad
\lambda_{\varphi}(p)\coloneqq T_p \Gamma_{d\varphi}.
\end{equation}
Here, our notation $\lambda_\infty(p)$ is different from that of \cite{KS90}, where the authors write $\lambda_0(p)$ for $T_p (T^*_{\pi(p)}X)$.
In this paper we do \emph{not} use the symbol $\lambda_0(p)$.
We briefly recall the definition of the inertia index of a triple of Lagrangian subspaces (see \cite[\S A.3]{KS90}).
Let $(E,\sigma)$ be a symplectic vector space and $\lambda_1,\lambda_2,\lambda_3$ be three Lagrangian subspaces of $E$.
We define a quadratic form $q$ on $\lambda_1 \oplus \lambda_2 \oplus \lambda_3$ by $q(v_1,v_2,v_3)=\sigma(v_1,v_2)+\sigma(v_2,v_3)+\sigma(v_3,v_1)$.
Then the \emph{inertia index} $\tau_{E}(\lambda_\infty,\lambda_1,\lambda_3)$ of the triple is defined as the signature of $q$.
Using the inertia index and the notation~\eqref{eq:notaionlambda}, one sets
\begin{equation}
\tau_{\varphi}=\tau_{p,\varphi}\coloneqq 
\tau_{\,T_p T^*X}(\lambda_\infty(p),\lambda_\Lambda(p),\lambda_\varphi(p)).
\end{equation}

\begin{proposition}[{\cite[Prop.~7.5.3]{KS90}}]\label{prp:maslovshft}
	For $i=1,2$, let $\varphi_i \colon X \to \bR$ be a $C^\infty$-function such that $\varphi_i(\pi(p))=0$ and $\Gamma_{d\varphi_i}$ intersects $\Lambda$ transversally at $p$.
	Let $F \in \Db(X)$ and assume that $\MS(F) \subset \Lambda$ in a neighborhood of $p$.
	Then
	\begin{equation}
	\RG_{\{\varphi_1 \ge 0\}}(F)_{\pi(p)}
	\simeq
	\RG_{\{\varphi_2 \ge 0\}}(F)_{\pi(p)}\left[\tfrac{1}{2}(\tau_{\varphi_2}-\tau_{\varphi_1})\right].
	\end{equation}
\end{proposition}

\begin{definition}[{\cite[Def.~7.5.4]{KS90}}]
	In the situation of \pref{prp:maslovshft}, $F$ is said to have microlocal type $L \in \Db(\Module(\bfk))$ with shift $d \in \tfrac{1}{2} \bZ$ at $p$ if
	\begin{equation}
	\RG_{\{\varphi \ge 0\}}(F)_{\pi(p)} 
	\simeq
	L\left[d-\tfrac{1}{2}\dim X-\tfrac{1}{2}\tau_\varphi\right]
	\end{equation}
	for some (hence for any) $C^\infty$-function $\varphi$ such that $\varphi(\pi(p))=0$ and $\Gamma_{d\varphi}$ intersects $\Lambda$ transversally at $p$.
	Moreover, if $L \simeq \bfk$, $F$ is said to be \emph{simple} along $\Lambda$ at $p$.
	If $F$ is simple at all points of $\Lambda$, one says that $F$ is simple along $\Lambda$.
\end{definition}

One can prove that if $F \in \Db(X)$ is simple along $\Lambda$, then $\mu hom(F,F)|_{\Lambda} \simeq \bfk_{\Lambda}$.
When $\Lambda$ is a conormal bundle to a closed submanifold $M$ of $X$ in a neighborhood of $p$, that is, $\pi|_{\Lambda} \colon \Lambda \to X$ has constant rank, then $F \in \Db(X)$ is simple along $\Lambda$ at $p$ if $F \simeq \bfk_M[d]$ in $\Db(X;p)$ for some $d \in \bZ$.

\begin{example}\label{eg:simpleshift}
	Let $X=\bR^{n+1}$ and consider the hyperplane $M=\bR^n \times \{0\}$.
	Then $\bfk_M$ is simple with shift $1/2$ along $T^*_MX$. 
\end{example}

We also recall the notion of quantized contact transformations.
Let $\chi \colon T^*X \supset \Omega_1 \simto \Omega_2 \subset T^*X$ be a contact transformation.
A \emph{quantized contact transformation} associated with $\chi$ is a kernel $K \in \Db(X \times X)$ which is simple along $(\id_X \times a)^{-1}\mathrm{Graph}(\chi)$ in $\Omega_2 \times \Omega_1^a$ and satisfies some properties (see \cite[\S7.2]{KS90} for details).
A quantized contact transformation $K$ induces an equivalence of categories
\begin{equation}
K \circ (\ast)
\colon
\Db(X;\Omega_1) \simto \Db(X;\Omega_2).
\end{equation}

\begin{proposition}[{\cite[Thm.~7.2.1]{KS90}}]\label{prp:qctmuhom}
	Let $K \in \Db(X \times X)$ be a quantized contact transformation associated with a contact transformation $\chi \colon T^*X \supset \Omega_1 \simto \Omega_2 \subset T^*X$.
	Moreover, let $F, G \in \Db(X;\Omega_1)$.
	Then 
	\begin{equation}
	\mu hom(K \circ F, K \circ G)|_{\Omega_2}
	\simeq
	\chi_* (\mu hom(F,G) |_{\Omega_1}).
	\end{equation}
\end{proposition}

The behavior of the shift of a simple sheaf under a quantized contact transformation is described by the inertia index.

\begin{proposition}[{\cite[Prop.~7.5.6 and Thm.~7.5.11]{KS90}}]\label{prp:qctshift}
	Let $F \in \Db(X)$ and assume that $F$ is simple with shift $d$ along $\Lambda$ at $p$.
	Let $\chi \colon T^*X \supset \Omega_1 \simto \Omega_2 \subset T^*X$ be a contact transformation defined in a neighborhood of $p$ and $K \in \Db(X \times X)$ be a quantized contact transformation associated with $\chi$.
	Assume that $K$ is simple with shift $d'$ along $(\id_X \times a)^{-1}\mathrm{Graph}(\chi)$ at $(\chi(p),p^a)$.
	Then $K \circ F$ is simple with shift $d+d'-\delta$ along $\chi(\Lambda)$ at $\chi(p)$, where
	\begin{equation}
	\delta
	\coloneqq 
	\frac{1}{2}\dim X
	+ \frac{1}{2}
	\tau(\lambda_\infty(p),\lambda_{\Lambda}(p),\chi^{-1}(\lambda_\infty(\chi(p)))).
	\end{equation}
\end{proposition}

\section{Sheaf quantization and Tamarkin's non-displaceability \allowbreak theorem}\label{sec:quantandnondisp}

In what follows, until the end of the paper, let $M$ be a non-empty compact connected manifold without boundary.

In this section, we review Tamarkin's approach to non-displaceability problems in symplectic geometry based on microlocal sheaf theory.
We also review sheaf quantization of Hamiltonian isotopies and compact exact Lagrangian submanifolds in cotangent bundles.

\subsection{Sheaf quantization of Hamiltonian isotopies (\cite{GKS})}\label{subsec:quantHam}

Guillermou--Kashiwara--Schapira~\cite{GKS} constructed sheaf quantizations of Hamiltonian isotopies.
Since the microsupports of sheaves are conic subsets of cotangent bundles, microlocal sheaf theory is related to the exact (homogeneous) symplectic structures rather than the symplectic structures of cotangent bundles.
In order to treat non-homogeneous Hamiltonian isotopies and non-conic Lagrangian submanifolds, an important trick is to add a variable and ``conify" them, which is an idea of Tamarkin's.

Denote by $(x;\xi)$ a local homogeneous symplectic coordinate system on $T^*M$ and by $(t;\tau)$ the homogeneous symplectic coordinate system on $T^*\bR$.
We set $\Omega_+\coloneqq \{\tau >0\} \subset T^*(M \times \bR)$ and define the map
\begin{equation}
\begin{split}
\xymatrix@R=10pt{
	\rho \colon \Omega_+ \ar[r] & T^*M \\
	\ (x,t;\xi,\tau) \ar@{|->}[r] \ar@{}[u]|-{\hspace{7pt} \rotatebox{90}{$\in$}} & (x;\xi/\tau). \ar@{}[u]|-{\rotatebox{90}{$\in$}}
}
\end{split}
\end{equation}
Let $I$ be an open interval in $\bR$ containing $0$ and $\phi=(\phi_s)_{s \in I} \colon T^*M \times I \to T^*M$ be a Hamiltonian isotopy with compact support.
Note that $\phi$ is the identity for $s=0$: $\phi_0=\id_{T^*M}$.
Then one can construct a homogeneous Hamiltonian isotopy $\wh{\phi} \colon \rT(M \times \bR) \times I \to \rT(M \times \bR)$ such that the following diagram commutes:
\begin{equation}\label{diag:homog}
\begin{split}
\xymatrix{
	\Omega_+ \times I \ar[r]^-{\wh{\phi}} \ar[d]_-{\rho \times \id} & \Omega_+ \ar[d]^-{\rho} \\
	T^*M \times I \ar[r]_-{\phi} & T^*M.
}
\end{split}
\end{equation}
Here $\wh{\phi}$ is called a homogeneous Hamiltonian isotopy if it is a Hamiltonian isotopy whose Hamiltonian function $\wh{H}$ is homogeneous of degree $1$: $\wh{H}_s(x,t;c\xi,c\tau)=c \cdot \wh{H}_s(x,t;\xi,\tau)$ for any $c>0$.
See \cite[Sect.~A.3]{GKS} for more details.
For simplicity, we set $N\coloneqq M \times \bR$ and consider a homogeneous Hamiltonian isotopy $\wh{\phi}=(\wh{\phi}_s)_s \colon \rT N \times I \to \rT N$ and the associated homogeneous Hamiltonian $\wh{H} \colon \rT N \times I \to \bR$.
We define a conic Lagrangian submanifold $\Lambda_{\wh{\phi}} \subset \rT N \times \rT N \times T^*I$ by
\begin{equation}\label{eq:deflambdahatphi}
\Lambda_{\wh{\phi}}\coloneqq 
\left\{\left(\wh{\phi}_s(y;\eta), (y;-\eta), (s;-\wh{H}_s \circ \wh{\phi}_s(y;\eta)) \right) \; \middle| \; (y;\eta) \in \rT N, s \in I \right\}.
\end{equation}
Note that
\begin{equation}
\Lambda_{\wh{\phi}} \circ T^*_sI
=
\left\{\left(\wh{\phi}_s(y;\eta), (y;-\eta)\right) \; \middle| \; (y;\eta) \in \rT N \right\}
\subset \rT N \times \rT N
\end{equation}
for any $s \in I$ (see \eqref{eq:compset} for the definition of $A \circ B$).

\begin{theorem}[{\cite[Thm.~4.3]{GKS}}]
	For a homogeneous Hamiltonian isotopy $\wh{\phi} \colon \rT N \times I \to \rT N$, there exists a unique object $K \in \Dlb(N \times N \times I)$ satisfying the following conditions:
	\begin{enumerate}
		\renewcommand{\labelenumi}{$\mathrm{(\arabic{enumi})}$}
		\item $\MS(K) \subset \Lambda_{\wh{\phi}} \cup T^*_{N \times N \times I}(N \times N \times I)$,
		\item $K|_{N \times N \times \{0\}} \simeq \bfk_{\Delta_N}$, where $\Delta_N$ is the diagonal of $N \times N$.
	\end{enumerate}
	Moreover $K$ is simple along $\Lambda_{\wh{\phi}}$ and both projections $\Supp(K) \to N \times I$ are proper.
\end{theorem}

The object $K$ is called the \emph{sheaf quantization} of $\wh{\phi}$.
For any $s \in I$, $\mathring{\MS}(K|_{N \times N \times \{s\}}) \subset \Lambda_{\wh{\phi}} \circ T^*_{s}I$ and $K|_{N \times N \times \{s\}}$ is a quantized contact transformation associated with $\wh{\phi}_{s} \colon \Omega_+ \simto \Omega_+$.

\subsection{Tamarkin's non-displaceability theorem (\cite{Tamarkin, GS14})}\label{subsec:Tamarkin}

A diffeomorphism $\psi \colon T^*M \to T^*M$ is said to be a Hamiltonian diffeomorphism if there exists a Hamiltonian isotopy with compact support $\phi=(\phi_s)_s \colon T^*M \times [0,1] \to T^*M$ such that $\phi_1=\psi$ and $\phi_0=\id_{T^*M}$.
Two  compact subsets $A$ and $B$ of $T^*M$ are said to be \emph{mutually non-displaceable} if $A \cap \psi(B) \neq \emptyset$ for any Hamiltonian diffeomorphism $\psi \colon T^*M \to T^*M$.
The non-displaceability problem is to determine whether or not given two compact subsets are mutually non-displaceable.
Tamarkin~\cite{Tamarkin} (see also Guillermou--Schapira~\cite{GS14}) considered some categories consisting of sheaves on $M \times \bR$ and deduced a criterion for non-displaceability using them.

We set $\Omega_+\coloneqq \{\tau >0 \} \subset T^*(M \times \bR)$ as before, where $(t;\tau)$ denotes the homogeneous symplectic coordinate system on $T^*\bR$.
We define the maps
\begin{equation}
\begin{split}
& \qquad \tilde{q}_1,\tilde{q}_2,s_\bR \colon M \times \bR \times \bR \lto M \times \bR, \\
\tilde{q}_1(x,t_1,t_2)&=(x,t_1), \ \tilde{q}_2(x,t_1,t_2)=(x,t_2), \ s_\bR(x,t_1,t_2)=(x,t_1+t_2).
\end{split}
\end{equation}
If there is no risk of confusion, we simply write $s$ for $s_\bR$.
We also define the involution
\begin{equation}
i \colon M \times \bR \to M \times \bR, \
(x,t) \longmapsto (x,-t).
\end{equation}

\begin{definition}
	For $F,G \in \Db(M \times \bR)$, one sets
	\begin{align}
	F \star G & \coloneqq Rs_!(\tilde{q}_1^{-1}F \otimes \tilde{q}_2^{\,-1}G), \\
	\cHom^\star(F,G)
	& \coloneqq R \tilde{q}_{1*} \cRHom(\tilde{q}_2^{\,-1}F,s^!G) \\
	& \ \simeq Rs_*\cRHom(\tilde{q}_2^{\, -1}i^{-1}F, \tilde{q}_1^!G).  \notag
	\end{align}
\end{definition}

Note that the functor $\star$ is a left adjoint to $\cHom^\star$.
The functor
\begin{equation}
\bfk_{M \times [0,+\infty)} \star (\ast) \colon \Db(M \times \bR) \lto \Db(M \times \bR)
\end{equation}
defines a projector on the left orthogonal ${}^\perp \Db_{\{\tau \le 0\}}(M \times \bR)$.
By using this projector, Tamarkin proved that the localized category $\Db(M \times \bR;\Omega_+)$ is equivalent to the left orthogonal ${}^\perp \Db_{\{\tau \le 0\}}(M \times \bR)$:
\begin{equation}
\Db(M \times \bR;\Omega_+)=\Db(M \times \bR)/\Db_{\{\tau \le 0\}}(M \times \bR) \simto {}^\perp \Db_{\{\tau \le 0\}}(M \times \bR).
\end{equation}

\begin{definition}[{\cite{Tamarkin}}]
	One defines
	\begin{equation}
		\cD(M)
		\coloneqq 
		\Db(M \times \bR;\Omega_+)
		\simeq 
		{}^\perp \Db_{\{\tau \le 0\}}(M \times \bR).
	\end{equation}
	For a compact subset $A$ of $T^*M$, one also defines a full subcategory $\cD_A(M)$ of $\cD(M)$ by
	\begin{equation}
		\cD_A(M)
		\coloneqq 
		\Db_{\rho^{-1}(A)}(M \times \bR;\Omega_+).
	\end{equation}
\end{definition}

For an object in $\cD(M)$, we take the canonical representative in ${}^\perp \Db_{\{\tau \le 0\}}(M \times \bR)$ via the projector unless otherwise specified.

\begin{proposition}[{\cite[Lem.~3.18]{GS14}}]\label{prp:morD}
	Let $F,G \in \cD(M)$.
	Then 
	\begin{equation}
	\Hom_{\cD(M)}(F,G) \simeq
	H^0 \RG_{M \times [0,+\infty)}(M \times \bR;\cHom^\star(F,G)).
	\end{equation}
\end{proposition}

The following separation theorem is due to Tamarkin~\cite{Tamarkin}.

\begin{theorem}[{\cite[Thm.~3.2]{Tamarkin} and \cite[Thm.~3.28]{GS14}}]\label{thm:separation}
	Let $A$ and $B$ be compact subsets of $T^*M$ and assume that $A \cap B=\emptyset$.
	Then for $F \in \cD_A(M)$ and $G \in \cD_B(M)$, one has $\Hom_{\cD(M)}(F,G) \simeq 0$.
\end{theorem}

\begin{proof}
	We give the outline of the proof.
	Denote by $t \colon M \times \bR \to \bR$ the function $(x,t) \mapsto t$.
	Recall the notation $\Gamma_{dt}=\{(x,t;0,1)\} \subset T^*(M \times \bR)$.
	Then one can show that
	\begin{equation}
	\Gamma_{dt} \cap \MS(\RG_{M \times [0,+\infty)}\cHom^\star(F,G))=\emptyset.
	\end{equation}
	Hence by \pref{prp:morD} and the microlocal Morse lemma (\pref{prp:microlocalmorse}), we have the conclusion.
\end{proof}

Using sheaf quantization of Hamiltonian isotopies, we can define Hamiltonian deformations in the category $\cD(M)$ as follows.
Let $\psi \colon T^*M \to T^*M$ be a Hamiltonian diffeomorphism and $\phi=(\phi_s)_s \colon T^*M \times I \to T^*M$ be a Hamiltonian isotopy with compact support satisfying $\psi=\phi_1$, where $I$ is an open interval containing the closed interval $[0,1]$.
Let $\wh{\phi} \colon \rT(M \times \bR) \times I \to \rT(M \times \bR)$ be the associated homogeneous Hamiltonian isotopy and $K \in \Dlb(M \times \bR \times M \times \bR \times I)$ be the sheaf quantization of $\wh{\phi}$.
Then the composition with $K_{1}\coloneqq K|_{M \times \bR \times M \times \bR \times \{1\}} \in \Db(M \times \bR \times M \times \bR)$ defines a functor
\begin{equation}\label{eq:functorPsi}
\Psi=K_{1} \circ (\ast) \colon
\Db(M \times \bR) \lto \Db(M \times \bR),
\end{equation}
which induces a functor $\Psi \colon \cD(M) \to \cD(M)$ (see \cite[Prop.~3.29]{GS14})\footnote{Although $\wh{\phi}$ does not satisfy \cite[(3.3)]{GKS} in general, $K|_{M \times \bR \times M \times \bR \times J}$ is bounded for any relatively compact subinterval $J$ of $I$. The author learned the detailed proof from S.~Guillermou. One can prove it using the properness of $\Supp(K) \to M \times \bR \times I$ and the fact that $K \simeq \sigma^{-1}K'$, where $K' \in \Dlb(M \times M \times \bR \times I)$ and $\sigma \colon M \times \bR \times M \times \bR \times I \to M \times M \times \bR \times I, (x,t,x',t',s) \mapsto (x,x',t-t',s)$. }.
Let $A$ be a compact subset of $T^*M$.
Then, for any $F \in \cD_A(M)$, \pref{prp:SScomp} and the commutative diagram \eqref{diag:homog} imply 
\begin{equation}
\MS(K_1 \circ F) \cap \Omega_+
\subset (\Lambda_{\wh{\phi}} \circ T^*_{1}I) \circ \rho^{-1}(A)
=\wh{\phi}_1(\rho^{-1}(A))
\subset \rho^{-1}(\psi(A)).
\end{equation}
Hence the functor also induces $\Psi \colon \cD_A(M) \to \cD_{\psi(A)}(M)$.

Tamarkin~\cite{Tamarkin} proved the non-displaceability theorem by using the category $\cD(M)$ and torsion objects, which we will explain below.
Moreover, Guillermou--Schapira~\cite{GS14} proved that torsion objects form a triangulated subcategory and introduced the quotient category $\cT(M)$, which is invariant under Hamiltonian deformations.
For $c \in \bR$, we define the translation map
\begin{equation}
T_c \colon M \times \bR \to M \times \bR, (x,t) \mapsto (x,t+c).
\end{equation}
For $F \in {}^\perp \Db_{\{\tau \le 0\}}(M \times \bR)$ and $c \in \bR_{\ge 0}$, there exists a canonical morphism $\tau_{0,c}(F) \colon F \to {T_c}_*F$.

\begin{definition}[{\cite{Tamarkin}}]
	An object $F \in {}^\perp \Db_{\{\tau \le 0\}}(M \times \bR)$ is said to be a \emph{torsion object} if $\tau_{0,c}(F)=0$ for some $c \ge 0$.
	Denote by $\cN_{\mathrm{tor}}$ the subcategory of torsion objects in ${}^\perp \Db_{\{\tau \le 0\}}(M \times \bR) \simeq \cD(M)$.
\end{definition}

Let $F \in {}^\perp \Db_{\{\tau \le 0\}}(M \times \bR)$ and assume that $\Supp(F) \subset M \times [a,b]$ for some compact interval $[a,b]$ of $\bR$.
Then $F$ is a torsion object.

\begin{proposition}[{\cite[Thm.~5.4]{GS14}}]
	The subcategory $\cN_{\mathrm{tor}}$ is a full triangulated subcategory of $\cD(M)$.
\end{proposition}

\begin{definition}[{\cite[Def.~5.6]{GS14}}]
	The triangulated category $\cT(M)$ is defined as the quotient category of $\cD(M)$ by $\cN_{\mathrm{tor}}$:
	$\cT(M)\coloneqq \cD(M)/\cN_{\mathrm{tor}}$.
\end{definition}

Hom spaces in $\cT(M)$ are described as inductive limits of those in $\cD(M)$.

\begin{proposition}[{\cite[Prop.~5.7]{GS14}}]\label{prp:morT}
	Let $F,G \in \cD(M)$.
	Then 
	\begin{equation}
	\varinjlim_{c \to +\infty} \Hom_{\cD(M)}(F,{T_c}_*G)
	\simto
	\Hom_{\cT(M)}(F,G).
	\end{equation}
\end{proposition}

The following is the Hamiltonian invariance theorem due to Tamarkin~\cite{Tamarkin}.

\begin{theorem}[{\cite[Thm.~3.9]{Tamarkin} and \cite[Thm.~6.1]{GS14}}]\label{thm:invham}
	Let $\psi \colon T^*M \to T^*M$ be a Hamiltonian diffeomorphism and $\Psi \colon \cD(M) \to \cD(M)$ be the functor associated with $\psi$.
	Then, for any $F \in \cD(M)$, one has
	\begin{equation}
	F \simeq \Psi(F) \quad \text{in $\cT(M)$}.
	\end{equation}
\end{theorem}

Combining \pref{thm:invham} with \pref{thm:separation} and \pref{prp:morT}, we can deduce the following non-displaceability theorem.

\begin{theorem}[{\cite[Thm.~3.1]{Tamarkin} and \cite[Cor.~6.3]{GS14}}]\label{thm:nondisp}
	Let $A$ and $B$ be compact subsets of $T^*M$.
	Assume that there exist $F \in \cD_A(M)$ and $G \in \cD_B(M)$ such that $\Hom_{\cT(M)}(F,G) \neq 0$.
	Then $A$ and $B$ are mutually non-displaceable.
\end{theorem}

\subsection{Guillermou's sheaf quantization of compact exact Lagrangian submanifolds (\cite{Gu12, Gulec})}\label{subsec:Guillermou}

Recall that a Lagrangian submanifold $L$ of $T^*M$ is said to be \emph{exact} if the restriction of the Liouville 1-form $\alpha|_L$ is exact.
Guillermou \cite{Gu12,Gulec} proved the existence of sheaf quantizations of compact exact Lagrangian submanifolds of $T^*M$.

Let $L$ be a compact connected exact Lagrangian submanifold of $T^*M$ and choose a primitive of the Liouville 1-form $f \colon L \to \bR$ satisfying $df=\alpha|_{L}$.
We define the \emph{conification} $\wh{L}_f \subset \Omega_+$ of $L$ with respect to $f$ by
\begin{equation}\label{eq:conification}
\wh{L}_f
\coloneqq 
\{
(x,t;\tau \xi,\tau) \mid
\tau >0, (x;\xi) \in L, t=-f(x;\xi)
\}.
\end{equation}
If there is no risk of confusion, we simply write $\wh{L}$ instead of $\wh{L}_f$.

Consider the category $\Db_{\wh{L} \cup T^*_{M \times \bR}(M \times \bR)}(M \times \bR)$ consisting of sheaves whose microsupports are contained in $\wh{L} \cup T^*_{M \times \bR}(M \times \bR)$.
By the compactness of $L$, there is $A \in \bR_{>0}$ such that $\wh{L} \subset T^*(M \times (-A,A))$.
Hence for any $F \in \Db_{\wh{L} \cup T^*_{M \times \bR}(M \times \bR)}(M \times \bR)$, the restrictions $F|_{M \times (-\infty,-A)}$ and $F|_{M \times (A,+\infty)}$ are locally constant.

\begin{definition}[{\cite[Def.~20.1]{Gu12} and \cite[Def.~13.1]{Gulec}}]
	Let $A \in \bR_{>0}$ satisfying $\wh{L} \subset T^*(M \times (-A,A))$.
	For $F \in \Db_{\wh{L} \cup T^*_{M \times \bR}(M \times \bR)}(M \times \bR)$, one defines $F_-,F_+ \in \Db(M)$ by
	\begin{equation}
	F_-\coloneqq F|_{M \times \{ -t\}},
	\quad
	F_+\coloneqq F|_{M \times \{ t\}}
	\end{equation}
	for any $t>A$ (independent of $t$).
	One also defines $\Db_{\wh{L} \cup T^*_{M \times \bR}(M \times \bR),+}(M \times \bR)$ as the full subcategory of $\Db_{\wh{L} \cup T^*_{M \times \bR}(M \times \bR)}(M \times \bR)$ consisting of $F$ such that $F_- \simeq 0$.
\end{definition}

Guillermou~\cite{Gu12,Gulec} proved the following existence and uniqueness of sheaf quantizations of compact exact Lagrangian submanifolds.

\begin{theorem}[{\cite[Thm.~26.1]{Gu12} and \cite[Thm.~18.1]{Gulec}}]\label{thm:Guiquan}
	Let $L, f$, and $\wh{L}=\wh{L}_f$ be as above.
	\begin{enumerate}
		\item For any rank $1$ locally constant sheaf $\cL \in \Module(\bfk_M)$, there exists an object $F \in \Db_{\wh{L} \cup T^*_{M \times \bR}(M \times \bR),+}(M \times \bR)$ satisfying $F_+ \simeq \cL$.
		\item Moreover $F$ in \emph{(i)} is unique up to a unique isomorphism and simple along $\wh{L}$.
	\end{enumerate}
\end{theorem}

We call the object $F \in \Db_{\wh{L} \cup T^*_{M \times \bR}(M \times \bR),+}(M \times \bR)$ in (i) the \emph{simple sheaf quantization} of $\wh{L}$ with respect to the rank $1$ locally constant sheaf $\cL$.
Moreover if $\cL$ is the constant sheaf $\bfk_M$, that is, $F_+ \simeq \bfk_M$, then $F$ is said to be the \emph{canonical sheaf quantization} of $\wh{L}$.
Note that the simple sheaf quantization of $\wh{L}$ with respect to $\cL$ is of the form $F \otimes q_M^{-1}
\cL$, where $F$ is the canonical sheaf quantization and $q_M \colon M \times \bR\to M$ is the projection.
We sometimes write a sheaf quantization associated with $L$ (and $f$) instead of $\wh{L}$ for simplicity.

\section{Intersections of compact exact Lagrangian submanifolds in cotangent bundles and sheaf quantization}\label{sec:intersection}

In this section we study intersections of compact exact Lagrangian submanifolds in cotangent bundles, using Tamarkin's category and Guillermou's sheaf quantizations.
In particular, we prove \pref{thm:intromain}, a Morse--Bott-type inequality for clean Lagrangian intersections.
Throughout this section,  for $i=1,2$ let $L_i$ be a compact connected exact Lagrangian submanifold and $f_i \colon L_i \to \bR$ be a primitive of the Liouville 1-form satisfying $df_i=\alpha|_{L_i}$.
We denote by $\Lambda_i\coloneqq \wh{L_i}$ the conification of $L_i$ with respect to $f_i$.
Moreover, let $F_i \in \Db_{\Lambda_i \cup T^*_{M \times \bR}(M \times \bR)}(M \times \bR)$ be a simple sheaf quantization of $\Lambda_i$.
Until the end of \pref{subsec:microlocalization-of-chomstar}, we do \emph{not} assume that $L_1$ and $L_2$ intersect cleanly.

\subsection{Non-displaceability of compact exact Lagrangian submanifolds}\label{subsec:nondisplagr}

In this subsection we prove that the Hom space in $\cT(M)$ between the canonical sheaf quantizations associated with compact exact Lagrangian submanifolds is isomorphic to the cohomology of the base manifold $M$.
Combined with \pref{thm:nondisp}, this implies the non-displaceability.

First we give a preliminary result useful to calculate Hom spaces in $\cD(M)$.

\begin{lemma}\label{lem:orthogonal}
	Let $L$ be a compact connected exact Lagrangian submanifold of $T^*M$ and $\Lambda=\wh{L}$ be the conification of $L$ with respect to some primitive.
	Then
	\begin{equation}
	\Db_{\Lambda \cup T^*_{M \times \bR}(M \times \bR),+}(M \times \bR) 
	\subset
	{}^\perp \Db_{\{\tau \le 0\}}(M \times \bR).
	\end{equation}
\end{lemma}

\begin{proof}
	By compactness, there exists a constant $B \in \bR$ such that $\Lambda \subset T^*(M \times (B,+\infty))$.
	Let $F \in \Db_{\Lambda \cup T^*_{M \times \bR}(M \times \bR),+}(M \times \bR)$ and $G \in \Db_{\{\tau \le 0\}}(M \times \bR)$.
	Since $\Lambda \subset \{\tau >0\}$, by \pref{prp:SSHom}, we have $\MS(\cRHom(F,G)) \subset \{\tau \le 0\}$.
	Applying the microlocal Morse lemma (\pref{prp:microlocalmorse}) to $\cRHom(F,G)$ and the function $t \colon M \times \bR \to \bR, (x,t) \mapsto t$, we get $\RHom(F,G) \simeq 0$ by the inclusion $\Supp(\cRHom(F,G)) \subset M \times [B,+\infty)$.
\end{proof}

\begin{proposition}\label{prp:cohlagr}
	Let $\cL_i\coloneqq (F_i)_+ \in \Module(\bfk_M)$ be the locally constant sheaf of rank $1$ associated with the simple sheaf quantization $F_i$ for $i=1,2$.
	Then there exists $c_0 \in \bR_{\ge 0}$ such that $\Hom_{\cD(M)}(F_2,{T_c}_*F_1[k])$ is isomorphic to $H^k(M;\cL_1 \otimes \cL_2^{\otimes -1})$ for any $c \ge c_0$ and $k \in \bZ$.
	In particular, 
	\begin{equation}\label{eq:lagmorT}
	\Hom_{\cT(M)}(F_2,F_1[k])
	\simeq
	H^k(M;\cL_1 \otimes \cL_2^{\otimes -1})
	\quad \text{for any $k \in \bZ$}.
	\end{equation}
\end{proposition}

\begin{proof}
	The proof is very similar to those of \cite[Thm.~20.3]{Gu12} and \cite[Thm.~13.3]{Gulec}.
	By \pref{lem:orthogonal}, for any $k \in \bZ$, we have
	\begin{equation}
	\Hom_{\cD(M)}(F_2,{T_c}_*F_1[k])
	=
	\Hom_{\Db(M \times \bR)}(F_2,{T_c}_*F_1[k]).
	\end{equation}
	By the compactness of $L_1$ and $L_2$, there exists $A \in \bR_{>0}$ satisfying $\Lambda_1, \Lambda_2 \subset T^*(M \times (-A,A))$.
	Take a sufficiently large $c_0 \in \bR_{\ge 0}$ such that $c_0>2A$.
	Then, by the isomorphism $F_2|_{M \times (A, +\infty)} \simeq \cL_2 \boxtimes \bfk_{(A, +\infty)}$ and the inclusion $\Supp({T_c}_*F_1) \subset M \times (c-A,+\infty)$, we get
	\begin{equation}
	\begin{split}
	\RHom(F_2,{T_c}_*F_1)
	& \simeq
	\RHom(\cL_2 \boxtimes \bfk_{\bR},{T_c}_*F_1) \\
	& \simeq
	\RG(M \times \bR;F_1 \otimes (\cL_2^{\otimes -1} \boxtimes \bfk_{\bR}))
	\end{split}
	\end{equation}
	for any $c \ge c_0$.
	Since $\MS(F_1 \otimes (\cL_2^{\otimes -1} \boxtimes \bfk_{\bR})) \subset \{\tau \ge 0\}$, we can apply the microlocal Morse lemma (\pref{prp:microlocalmorse}) and obtain
	\begin{equation}
	\begin{split}
	\RG(M \times \bR;F_1 \otimes (\cL_2^{\otimes -1} \boxtimes \bfk_{\bR}))
	& \simeq
	\RG(M \times (A,+\infty);F_1 \otimes (\cL_2^{\otimes -1} \boxtimes \bfk_{\bR})) \\
	& \simeq
	\RG(M \times (A,+\infty);(\cL_1 \otimes \cL_2^{\otimes -1}) \boxtimes \bfk_{\bR}) \\
	& \simeq
	\RG(M;\cL_1 \otimes \cL_2^{\otimes -1}).
	\end{split}
	\end{equation}
	The second assertion follows from \pref{prp:morT}.
\end{proof}

\begin{remark}
	In the special case where both $L_1$ and $L_2$ are the zero-section $T^*_MM$ of $T^*M$, \eqref{eq:lagmorT} was already obtained by Guillermou--Schapira~\cite{GS14}.
	The outline of the proof is as follows.
	The simple sheaf quantization associated with the zero-section $T^*_MM$ and a rank $1$ locally constant sheaf $\cL \in \Module(\bfk_M)$ is isomorphic to $\cL \boxtimes \bfk_{[0,+\infty)}$.
	In \cite{GS14}, Guillermou and Schapira proved that the functor
	\begin{equation}
	\Db(M) \lto \cT(M), \
	F \longmapsto F \boxtimes \bfk_{[0,+\infty)}
	\end{equation}
	is fully faithful (see \cite[Cor.~5.8]{GS14}).
	We thus obtain
	\begin{equation}\label{eq:special}
	\begin{split}
	\Hom_{\cT(M)}(\cL_2 \boxtimes \bfk_{[0,+\infty)},\cL_1 \boxtimes \bfk_{[0,+\infty)}[k])
	& \simeq
	\Hom_{\Db(M)}(\cL_2,\cL_1[k]) \\
	& \simeq
	H^k(M;\cL_1 \otimes \cL_2^{\otimes -1})
	\end{split}
	\end{equation}
	for rank $1$ locally constant sheaves $\cL_1, \cL_2 \in \Module(\bfk_M)$.
	
	Moreover, we can prove \eqref{eq:lagmorT} for general compact exact Lagrangians $L_1$ and $L_2$ using \eqref{eq:special} and \pref{prp:isomT} below.
	The following was pointed out to the author by T.~Kuwagaki.
	
	\begin{proposition}\label{prp:isomT}
		Let $L$ be a compact connected exact Lagrangian submanifold of $T^*M$.
		Let $\cL \in \Module(\bfk_M)$ be a locally constant sheaf of rank $1$ and let $F \in \Db_{\wh{L} \cup T^*_{M \times \bR}(M \times \bR),+}(M \times \bR)$ be the simple sheaf quantization associated with $L$ satisfying $F_+ \simeq \cL$.
		Then 
		\begin{equation}
		F \simeq \cL \boxtimes \bfk_{[0,+\infty)}
		\quad \text{in $\cT(M)$}.
		\end{equation}
	\end{proposition}
	
	\begin{proof}
		By the compactness of $L$, we can take a sufficiently large $A \in \bR_{>0}$ such that $\wh{L} \subset T^*(M \times (-A,A))$.
		Since $F|_{M \times (A,+\infty)} \simeq \cL \boxtimes \bfk_{(A,+\infty)}$, there exists a canonical morphism
		\begin{equation}\label{eq:morconst}
		F \lto \cL \boxtimes \bfk_{[A+1,+\infty)}.
		\end{equation}
		The cone of this morphism is supported in $M \times [-A,A+1]$ and hence a torsion object.
		Therefore the morphism \eqref{eq:morconst} is an isomorphism in $\cT(M)$.
		A similar argument shows that the morphism $\cL \boxtimes \bfk_{[0,+\infty)} \to \cL \boxtimes \bfk_{[A+1,+\infty)}$ is an isomorphism in $\cT(M)$.
	\end{proof}
\end{remark}

By \pref{thm:nondisp} and \pref{prp:cohlagr} we obtain the following:

\begin{corollary}\label{cor:nondislagr}
	In the same notation as in \pref{prp:cohlagr}, assume that $F_i$ is the canonical sheaf quantization of $\wh{L_i}$, that is, $\cL_i \simeq \bfk_M$ for $i=1,2$.
	Then 
	\begin{equation}
	\Hom_{\cT(M)}(F_2,F_1[k])
	\simeq
	H^k(M;\bfk)
	\quad \text{for any $k \in \bZ$}.
	\end{equation}
	In particular, $L_1$ and $L_2$ are mutually non-displaceable.
\end{corollary}

\subsection{Morse--Bott inequality for $\cHom^\star$}\label{subsec:M-B for Hom}

In this subsection, we shall apply the Morse--Bott inequality for sheaves to $\cHom^\star(F_2,F_1)$.
For this purpose, we estimate $\MS(\cHom^\star(F_2,F_1))$.
Recall the isomorphism
\begin{equation}
\cHom^\star(F_2,F_1) \simeq R{s}_* \cRHom(\tilde{q}_2^{-1}i^{-1}F_2,\tilde{q}_1^!F_1),
\end{equation}
where $\tilde{q}_1, \tilde{q}_2 \colon M \times \bR \times \bR \to M \times \bR$ are the projections, $s \colon M \times \bR \times \bR \to M \times \bR$ is the addition map, and $i \colon M \times \bR \to M \times \bR$ is the involution $(x,t) \mapsto (x,-t)$.
Since $\tilde{q}_2$ and $\tilde{q}_1$ are submersions, by \pref{thm:operations}(ii) we have inclusions
\begin{equation}
\begin{split}
\mathring{\MS}(\tilde{q}_2^{-1}i^{-1}F_2)
& \subset \tilde{q}_{2d} \tilde{q}_{2\pi}^{-1}\mathring{\MS}(i^{-1}F_2) \\
& =
\left\{(x,t_1,t_2;\tau_2 \xi_2,0,-\tau_2)
\; \middle| \;
\begin{split}
& \tau_2>0, (x;\xi_2) \in L_2, \\
& t_1 \in \bR, t_2=f_2(x;\xi_2)
\end{split}
\right\}
\end{split}
\end{equation}
and
\begin{equation}
\begin{split}
\mathring{\MS}(\tilde{q}_1^!F_1)
& \subset \tilde{q}_{1d} \tilde{q}_{1\pi}^{-1}\mathring{\MS}(F_1) \\
& =
\left\{(x,t_1,t_2;\tau_1 \xi_1,\tau_1,0)
\; \middle| \;
\begin{split}
& \tau_1>0, (x;\xi_1) \in L_1, \\
& t_1=-f_1(x;\xi_1), t_2 \in \bR
\end{split}
\right\}.
\end{split}
\end{equation}
Hence $\mathring{\MS}(\tilde{q}_2^{-1}i^{-1}F_2) \cap \mathring{\MS}(\tilde{q}_1^!F_1) =\emptyset$, and by \pref{prp:SSHom} we obtain
\begin{equation}\label{eq:estSSH}
\begin{split}
& \mathring{\MS}(\cRHom(\tilde{q}_2^{-1}i^{-1}F_2,\tilde{q}_1^!F_1))
\subset \mathring{\MS}(\tilde{q}_2^{-1}i^{-1}F_2)^a + \mathring{\MS}(\tilde{q}_1^!F_1) \\
& =
\left\{
(x,t_1,t_2;\tau_1\xi_1-\tau_2\xi_2,\tau_1,\tau_2)
\; \middle| \;
\begin{split}
& \tau_1, \tau_2>0, \\
& (x;\xi_1) \in L_1, (x;\xi_2) \in L_2, \\
& t_1=-f_1(x;\xi_1), t_2=f_2(x;\xi_2)
\end{split}
\right\} \\
& =:
\Lambda_{M \times \bR \times \bR}.
\end{split}
\end{equation}

\begin{lemma}\label{lem:conicclosure}
	One has
	\begin{equation}
	\begin{split}
	&
	v_d^{-1}\left(\overline{v_\pi(\Lambda_{M \times \bR \times \bR} \cup T^*_{M \times \bR \times \bR}(M \times \bR \times \bR))}\right) \\
	= 	\ &
	v_d^{-1}v_\pi(\Lambda_{M \times \bR \times \bR} \cup T^*_{M \times \bR \times \bR}(M \times \bR \times \bR)) \\
	= \ &
	{s}_\pi {s}_d^{-1}(\Lambda_{M \times \bR \times \bR} \cup T^*_{M \times \bR \times \bR}(M \times \bR \times \bR)).
	\end{split}
	\end{equation}
	In other words,
	\begin{equation}
	\begin{split}
	& s_\sharp(\Lambda_{M \times \bR \times \bR} \cup T^*_{M \times \bR \times \bR}(M \times \bR \times \bR)) \\
	= \ &
	{s}_\pi {s}_d^{-1}(\Lambda_{M \times \bR \times \bR} \cup T^*_{M \times \bR \times \bR}(M \times \bR \times \bR)).
	\end{split}
	\end{equation}
\end{lemma}

See \pref{subsec:nonproper} for the notation $v_\pi, v_d$, and $s_\sharp$ associated with the constant linear map $s \colon M \times \bR \times \bR \to M \times \bR$.

\begin{proof}
	Define $\Lambda' \subset T^*M \times \bR \times (\bR \times \bR)$ by
	\begin{equation}
	\Lambda'\coloneqq 
	\left\{
	((x;\tau_1\xi_1-\tau_2\xi_2),(t;\tau_1,\tau_2))
	\; \middle| \;
	\begin{aligned}
	& \tau_1, \tau_2>0, (x;\xi_1) \in L_1, (x;\xi_2) \in L_2, \\
	& t=f_2(x;\xi_2)-f_1(x;\xi_1)
	\end{aligned}
	\right\}.
	\end{equation}
	Then the set $v_\pi(\Lambda_{M \times \bR \times \bR} \cup T^*_{M \times \bR \times \bR}(M \times \bR \times \bR))$ is equal to $\Lambda' \cup (T^*_MM \times \bR \times \{(0,0)\}) \subset T^*M \times \bR \times (\bR \times \bR)$.
	It suffices to check that $\Lambda' \cup (T^*_MM \times \bR \times \{(0,0)\})$ is equal to its closure.
	By the compactness of $L_1$ and $L_2$, there exists $C \in \bR_{>0}$ such that $|\xi| \le C(|\tau_1|+|\tau_2|)$ for any $((x;\xi),(t;\tau_1,\tau_2)) \in \Lambda'$.
	Therefore the same inequality holds on the closure $\overline{\Lambda'}$ of $\Lambda'$.
	Hence if $((x;\xi), (t;\tau_1,\tau_2)) \in \overline{\Lambda'}$ and $\tau_1=\tau_2=0$ then $\xi=0$, which proves the equality.
\end{proof}

By \pref{prp:nonproper}, \pref{lem:conicclosure}, and \eqref{eq:estSSH}, $\mathring{\MS}(\cHom^\star(F_2,F_1))$ is estimated as
\begin{equation}\label{eq:SSestHomstar}
\begin{aligned}
\mathring{\MS}(\cHom^\star(F_2,F_1))
&\subset
s_\sharp(\Lambda_{M \times \bR \times \bR} \cup T^*_{M \times \bR \times \bR}(M \times \bR \times \bR)) \cap \rT(M \times \bR) \\
& =
{s}_\pi {s}_d^{-1}(\Lambda_{M \times \bR \times \bR} \cup T^*_{M \times \bR \times \bR}(M \times \bR \times \bR)) \cap \rT(M \times \bR) \\
& \subset
\left\{
(x,t;\tau(\xi_1-\xi_2),\tau) \; \middle| \;
\begin{aligned}
& \tau >0, \\
& (x;\xi_1) \in L_1, (x;\xi_2) \in L_2, \\
& t=f_2(x;\xi_2)-f_1(x;\xi_1)
\end{aligned}
\right\} \\
& =:
\Lambda_{M \times \bR}.
\end{aligned}
\end{equation}
Let $t \colon M \to \bR$ be the function $(x,t) \mapsto t$.
Then, by \eqref{eq:SSestHomstar}, we obtain
\begin{equation}
\Gamma_{dt} \cap \MS(\cHom^\star(F_2,F_1))  
\subset
\left\{
(x,t;0,1)
\; \middle| \;
\begin{aligned}
& \exists\, (x;\xi) \in L_1 \cap L_2, \\
& t=f_2(x;\xi)-f_1(x;\xi)
\end{aligned}
\right\}.
\end{equation}
By this inclusion, we find that $\RG_{M \times [c,+\infty)}(\cHom^\star(F_2,F_1))|_{M \times \{ c \}} \simeq 0$ if $c \not\in \{f_2(p)-f_1(p) \mid p \in L_1 \cap L_2 \}$.

\begin{proposition}\label{prp:Homstarineq}
	Let $a,b \in \bR$ with $a<b$ or $a \in \bR, b=+\infty$.
	Assume that 
	\begin{enumerate}
		\renewcommand{\labelenumi}{$\mathrm{(\arabic{enumi})}$}
		\item the point $a \in \bR$ is not an accumulation point of $\{ f_2(p)-f_1(p) \mid p \in L_1 \cap L_2 \} \subset \bR$,
		\item the set $\{ f_2(p)-f_1(p) \mid p \in L_1 \cap L_2 \} \cap [a,b)$ is finite,
		\item the object $\RG( M \times \{c\};\RG_{M \times [c,+\infty)}(\cHom^\star(F_2,F_1))|_{M \times \{ c \}} )$ has finite-dimensional cohomology for any $a \le c<b$.
	\end{enumerate}
	Then 
	\begin{equation}\label{eq:Homstarineq}
	\begin{split}
	& \sum_{a \le c<b} \dim H^k \RG \left( M \times \{c\};\RG_{M \times [c,+\infty)}(\cHom^\star(F_2,F_1))|_{M \times \{ c \}} \right)  \\
	& \qquad \ge
	\dim H^k \RG_{M \times [a,b)}(M \times (-\infty,b);\cHom^\star(F_2,F_1))
	\end{split}
	\end{equation}
	for any $k \in \bZ$.
\end{proposition}

\begin{proof}
	We set $H\coloneqq \cHom^\star(F_2,F_1)$.
	By assumption~(1) we can take $a'<a$ such that
	\begin{equation}\label{eq:assumption}
	f_1(p)-f_2(p) \not\in [a',a) \quad \text{for any $p \in L_1 \cap L_2$.}
	\end{equation}
	By \eqref{eq:SSestHomstar} and \eqref{eq:assumption} we have $\mathring{\MS}(H) \cap \mathring{\MS}(\bfk_{M \times [a',+\infty)})=\emptyset$.
	Hence by \pref{prp:SSHom} we obtain
	\begin{equation}\label{eq:SSestprecise}
	\begin{split}
	\mathring{\MS}(\RG_{M \times [a',+\infty)}(H))
	& =
	\mathring{\MS}(\cRHom(\bfk_{M \times [a',+\infty)},H)) \\
	& \subset
	\Lambda_{M \times \bR} \cap \pi^{-1}(\{t>a' \})+\{(x,a';0,-\tau') \mid \tau' > 0 \}.
	\end{split}
	\end{equation}
	Set $H'\coloneqq \RG_{M \times [a',+\infty)}(H)|_{M \times (-\infty,b)} \in \Db(M \times (-\infty,b))$ and let $t \colon M \times (-\infty,b) \to \bR$ be the function $(x,t) \mapsto t$.
	We shall apply the Morse--Bott inequality for sheaves (\pref{thm:morsebottineq}) to $H'$ and $t \colon M \times (-\infty,b) \to \bR$.
	Combining \eqref{eq:SSestHomstar} with \eqref{eq:SSestprecise}, we get
	\begin{equation}
	\Gamma_{dt} \cap \MS(H')
	\subset
	\{
	(x,t;0,1) \mid
	\exists\, p \in L_1 \cap L_2,
	x=\pi(p),
	a'<t=f_2(p)-f_1(p)<b
	\}.
	\end{equation}
	Hence the conditions in \pref{thm:morsebottineq} are satisfied by \eqref{eq:assumption} and assumptions~(2) and (3).
	Hence we have the inequality
	\begin{equation}\label{eq:ineqpert}
	\begin{split}
	& \sum_{a' <c<b} \dim H^k \RG \left( M \times \{c\};\RG_{M \times [c,+\infty)}(H)|_{M \times \{ c \}} \right)  \\
	& \qquad \ge
	\dim H^k \RG_{M \times [a',b)}(M \times (-\infty,b);H)
	\end{split}
	\end{equation}
	for any $k \in \bZ$.
	Moreover, by \eqref{eq:SSestHomstar}, \eqref{eq:assumption}, and \eqref{eq:SSestprecise}, we get $\Gamma_{dt} \cap \MS(H') \cap \pi^{-1}(M \times [a',a))=\emptyset$.
	Applying the microlocal Morse lemma (\pref{prp:microlocalmorse}), we have
	\begin{equation}
	\begin{split}
	\RG_{M \times [a',a)}(M \times (-\infty,a);H)
	& \simeq
	\RG(M \times (-\infty,a);H') \\
	& \simeq
	\RG((-\infty,a');H') \simeq 0.
	\end{split}
	\end{equation}
	Thus we get $\RG_{M \times [a,b)}(M \times (-\infty,b);H) \simeq \RG_{M \times [a',b)}(M \times (-\infty,b);H)$.
	On the other hand, by \eqref{eq:assumption}, $\RG_{M \times [c,+\infty)}(H)|_{M \times \{ c \}} \simeq 0$ for $c \in [a',a)$ and the left-hand side of \eqref{eq:ineqpert} is equal to that of \eqref{eq:Homstarineq}.
	This completes the proof.
\end{proof}

\begin{remark}
	C.~Viterbo announced that he had found some relation between the section of $\cHom^\star(F_2,F_1)$ on $M \times (-\infty,\lambda)$ and the Floer cohomology complex $\mathrm{CF}_{<\lambda}(L_2,L_1)$ filtered by $\{p \in L_1 \cap L_2 \mid f_2(p)-f_1(p)<\lambda \}$.
	Inspired by his work, in \pref{prp:Homstarineq} we consider not only the section on $M \times \bR$ but also that on $M \times (-\infty,b)$ .
\end{remark}

\subsection{Microlocalization of $\cHom^\star$}\label{subsec:microlocalization-of-chomstar}

In this subsection we describe $\RG(M \times \{c\}; \RG_{M \times [c,+\infty)}(\cHom^\star(F_2,F_1))|_{M \times \{c\}})$ in terms of the functor $\mu hom$.
Applying ${T_{c}}_*$ to $F_2$, we may assume $c=0$.
The following lemma follows from \pref{prp:microlocalcoh}.

\begin{lemma}\label{lem:microsection}
	Set $V_+\coloneqq \{(x,0;0,\tau) \mid \tau >0\} \subset T^*_{M \times \{0\}}(M \times \bR)$.
	Then 
	\begin{equation}
	\begin{split}
	& \RG(M \times \{0\};\RG_{M \times [0,+\infty)}(\cHom^\star(F_2,F_1))|_{M \times \{0\}}) \\
	& \qquad \simeq
	\RG(V_+;\mu_{M \times \{0\}}(\cHom^\star(F_2,F_1))|_{V_+}).
	\end{split}
	\end{equation}
\end{lemma}

Recall the isomorphism
\begin{equation}
\cHom^\star(F_2,F_1)
\simeq
Rs_* \delta^!
\cRHom(q_2^{-1}i^{-1}F_2,q_1^{!}F_1),
\end{equation}
where $s \colon M \times \bR \times \bR \to M \times \bR$ is the addition map, $\delta \colon M \times \bR \times \bR \to M \times M \times \bR \times \bR$ is the diagonal embedding, and $q_i \colon M \times \bR \times M \times \bR \to M \times \bR$ is the $i$th projection.
The morphism $s$ induces the following commutative diagram, where we omit $T^*M$ (resp.\ $T^*_MM$) in the first (resp.\ second) row and use the same symbol $s$ for the addition map $\bR \times \bR \to \bR$:
\begin{equation}
\begin{split}
\xymatrix{
	T^*(\bR \times \bR) & (\bR \times \bR) \times_\bR T^*\bR \ar[l]_-{s_d} \ar[r]^-{s_\pi} & T^*\bR \\
	T^*_{s^{-1}(0)}(\bR \times \bR) \ar@{^{(}->}[u] & s^{-1}(0) \times_{\{0\}} T_0^*\bR \ar@{^{(}->}[u] \ar[l]_-{\sim} \ar[r]_-{s_\pi} & T_0^*\bR. \ar@{^{(}->}[u]
}
\end{split}
\end{equation}
We denote by $\pi_s \colon T^*_MM \times T^*_{s^{-1}(0)}(\bR \times \bR) \to T^*_MM \times T^*_0 \bR \simeq T^*_{M \times \{0\}}(M \times \bR)$ the induced morphism in the second row in the above diagram.
On the other hand, the morphism $\delta$ induces the following commutative diagram, where we omit $T^*_{s^{-1}(0)}(\bR \times \bR)$:
\begin{equation}
\begin{split}
\xymatrix{
	T^*M & M \times_{M \times M} T^*(M \times M) \ar[l]_-{\delta_d} \ar[r]^-{\delta_\pi} & T^*(M \times M) \\
	T^*_{M}M \ar@{^{(}->}[u] & M \times_{\Delta_M} T_{\Delta_M}^*(M \times M) \ar@{^{(}->}[u] \ar[l] \ar[r]^-{\sim} & T_{\Delta_M}^*(M \times M) \ar@{^{(}->}[u] \\
	M \ar@{=}[u] & T^*M \ar[l]^-{\pi_M} \ar@{=}[r] \ar@{=}[u]& T^*M. \ar@{=}[u]
}
\end{split}
\end{equation}
Moreover, let $\iota \colon T^*\bR \simeq T^*_{\Delta_\bR}(\bR \times \bR) \simto T^*_{s^{-1}(0)}(\bR \times \bR)$ be the isomorphism of line bundles defined by $(t_1,t_2,\tau, -\tau) \mapsto (t_1,-t_2,\tau,\tau)$.
We also use the same symbol~$\iota$ for the induced isomorphism $T^*(M \times \bR) \simeq T^*M \times T^*_{\Delta_\bR}(\bR \times \bR) \simto T^*M \times T^*_{s^{-1}(0)}(\bR \times \bR)$.

\begin{proposition}\label{prp:Homstarmuhom}
	Set $V_+\coloneqq \{(x,0;0,\tau) \mid \tau >0\} \subset T^*_{M \times \{0\}}(M \times \bR)$ as in \pref{lem:microsection} and define $\pi'\coloneqq \pi_s \circ \pi_M \circ \iota \colon T^*(M \times \bR) \to T^*_{M \times \{0\}}(M \times \bR)$.
	Then 
	\begin{equation}
	\mu_{M \times \{0\}}(\cHom^\star(F_2,F_1))|_{V_+}
	\simeq
	\left.
	\left( R\pi'_* \mu hom(F_2,F_1) \right) \right|_{V_+}.
	\end{equation}
\end{proposition}

\begin{proof}
	(a)
	Set $H=\cHom^\star(F_2,F_1)$.
	First, we note that $\mu_{M \times \{0\}}(H) \simeq \mu_{M \times \{0\}}(H|_{M \times (-1,1)})$.
	Set $U\coloneqq M \times (-1,1) \subset M \times \bR$.
	There exists a sufficiently large $A \in \bR_{>0}$ such that $F_1$ and $F_2$ are locally constant on $M \times (A-2,+\infty)$.
	Since the problem is local on $M$, we may assume that $F_1$ and $F_2$ are constant on $M \times (A-2,+\infty)$ from the beginning.
	Then $\tilde{q}_1^{!}F_1 \simeq \tilde{q}_1^{-1}F_1[1]$ is constant on $s^{-1}(U) \cap (M \times \bR \times (-\infty,-A+1))$, which implies isomorphisms
	\begin{equation}
	\begin{split}
	& \cRHom(\tilde{q}_2^{-1}i^{-1}\bfk_{M \times [A,+\infty)},\tilde{q}_1^!F_1)|_{s^{-1}(U)} \\
	\simeq \, &
	\cRHom(\bfk_{M \times \bR \times (-\infty,-A]},\bfk_{M \times \bR \times \bR}[1])|_{s^{-1}(U)} \\
	\simeq \, &
	\RG_{s^{-1}(U) \cap (M \times \bR \times (-\infty,-A])}(\bfk_{s^{-1}(U)})[1].
	\end{split}
	\end{equation}
	Therefore we obtain
	\begin{equation}
	(Rs_*\cRHom(\tilde{q}_2^{-1}i^{-1}\bfk_{M \times [A,+\infty)}, \tilde{q}_1^!F_1))|_{U}
	\simeq 0.
	\end{equation}
	By the distinguished triangle
	\begin{equation}\label{eq:exctri}
	F'_2 \lto F_2 \lto \bfk_{M \times [A,+\infty)} \toone
	\end{equation}
	with $F'_2$ supported in some compact subset, we find that
	\begin{equation}\label{eq:reductionproper}
	(Rs_*\cRHom(\tilde{q}_2^{-1}i^{-1}F_2,\tilde{q}_1^!F_1))|_{U}
	\simeq
	(Rs_*\cRHom(\tilde{q}_2^{-1}i^{-1}F'_2,\tilde{q}_1^!F_1))|_{U}
	\end{equation}
	and $s$ is proper on $\Supp(\cRHom(\tilde{q}_2^{-1}i^{-1}F'_2,\tilde{q}_1^!F_1))$.
	
	\noindent(b)
	Since $s$ is proper on the support, by \pref{prp:mufunctorial}(i), we have
	\begin{equation}
	\mu_{M \times \{0\}}(Rs_*\cRHom(\tilde{q}_2^{-1}i^{-1}F'_2,\tilde{q}_1^!F_1))
	\simeq
	R{\pi_s}_* \mu_{M \times s^{-1}(0)}(\cRHom(\tilde{q}_2^{-1}i^{-1}F'_2,\tilde{q}_1^!F_1)).
	\end{equation}
	Moreover, since $\delta$ is non-characteristic for $\MS(\cRHom(q_2^{-1}i^{-1}F'_2, q_1^!F_1))$ and $\delta|_{M \times s^{-1}(0)} \colon M \times s^{-1}(0) \to \Delta_M \times s^{-1}(0)$ is a submersion, by \pref{prp:mufunctorial}(ii) we obtain
	\begin{equation}
	\begin{split}
	\mu_{M \times s^{-1}(0)}(\cRHom(\tilde{q}_2^{-1}i^{-1}F'_2,\tilde{q}_1^!F_1))
	& \simeq
	\mu_{M \times s^{-1}(0)}(\delta^! \cRHom(q_2^{-1}i^{-1}F'_2,q_1^!F_1)) \\
	& \simeq
	R {\pi_M}_* \mu_{\Delta_M \times s^{-1}(0)}  \cRHom(q_2^{-1}i^{-1}F'_2,q_1^!F_1).
	\end{split}
	\end{equation}
	Let $i_2 \colon M \times \bR \times \bR \to M \times \bR \times \bR$ be the involution $(x,t_1,t_2) \mapsto (x,t_1,-t_2)$.
	Note that the associated automorphism of $T^*M \times T^*(\bR \times \bR)$ induces $\iota \colon T^*M \times T^*_{\Delta_\bR}(\bR \times \bR) \simto T^*M \times T^*_{s^{-1}(0)}(\bR \times \bR)$.
	Then, by \pref{prp:mufunctorial}(i) again, we have
	\begin{equation}
	\begin{split}
	\mu_{{\Delta_M} \times s^{-1}(0)}
	\cRHom(i_2^{-1}q_2^{-1}F'_2,q_1^!F_1)
	& \simeq
	\mu_{{\Delta_M} \times s^{-1}(0)}
	{i_2}_* \cRHom(q_2^{-1}F'_2,q_1^!F_1) \\
	& \simeq
	\iota_* \mu_{\Delta_{M \times \bR}}
	\cRHom(q_2^{-1}F'_2,q_1^!F_1) \\
	& \simeq
	\iota_* \mu hom(F'_2,F_1).
	\end{split}
	\end{equation}
	
	\noindent (c)
	By \pref{prp:muhomss} we have
	\begin{equation}
	\Supp(\mu hom(\bfk_{M \times [A,+\infty)},F_1))
	\subset
	T^*_{M \times \bR}(M \times \bR).
	\end{equation}
	Thus, by the distinguished triangle \eqref{eq:exctri}, we get
	\begin{equation}
	\mu hom(F'_2,F_1)|_{\{\tau >0\}} \simto \mu hom(F_2,F_1)|_{\{\tau >0\}},
	\end{equation}
	which completes the proof.
\end{proof}

We define an open subset $\Omega_+$ of $T^*(M \times \bR) \simeq T^*M \times T^*\bR$ by $\Omega_+\coloneqq \{\tau >0 \} \subset T^*(M \times \bR)$.
Combining \pref{prp:Homstarineq} with \pref{lem:microsection} and \pref{prp:Homstarmuhom}, we obtain the following:

\begin{proposition}\label{prp:ineqmuhomver}
	Let $a,b \in \bR$ with $a<b$ or $a \in \bR, b=+\infty$.
	Assume
	\begin{enumerate}
		\renewcommand{\labelenumi}{$\mathrm{(\arabic{enumi})}$}
		\item the point $a \in \bR$ is not an accumulation point of $\{ f_2(p)-f_1(p) \mid p \in L_1 \cap L_2 \} \subset \bR$,
		\item the set $\{ f_2(p)-f_1(p) \mid p \in L_1 \cap L_2 \} \cap [a,b) \subset \bR$ is finite,
		\item the object $\RG \left( \Omega_+;\mu hom({T_{c}}_*F_2,F_1)|_{\Omega_+} \right)$ has finite-dimensional cohomology for any $a \le c <b$.
	\end{enumerate}
	Then 
	\begin{equation}
	\begin{split}
	& \sum_{a \le c<b} \dim H^k \RG \left( \Omega_+;\mu hom({T_{c}}_*F_2,F_1)|_{\Omega_+} \right) \\
	& \qquad \ge
	\dim H^k \RG_{M \times [a,b)}(M \times (-\infty,b);\cHom^\star(F_2,F_1))
	\end{split}
	\end{equation}
	for any $k \in \bZ$.
\end{proposition}

\subsection{Clean intersections of compact exact Lagrangian submanifolds}\label{subsec:clean}

Throughout this subsection we assume the following:

\begin{assumption}\label{as:assumptionclean}
	The Lagrangian submanifolds $L_1$ and $L_2$ intersect cleanly, that is, $L_1 \cap L_2$ is a submanifold of $T^*M$ and $T_p(L_1 \cap L_2)=T_pL_1 \cap T_pL_2$ for any $p \in L_1 \cap L_2$.
\end{assumption}

Under the assumption, the intersection $L_1 \cap L_2$ has finitely many connected components, which are compact submanifolds of $T^*M$, and the value $f_2(p)-f_1(p)$ is constant on each component.
In particular, the set $\{ f_2(p)-f_1(p) \mid p \in L_1 \cap L_2 \} \subset \bR$ is finite.
For a component $C$ of $L_1 \cap L_2$, we define $f_{21}(C)\coloneqq f_2(p)-f_1(p)$, taking some $p \in C$.

Under \pref{as:assumptionclean}, we shall compute $\mu hom({T_{c}}_*F_2,F_1)|_{\Omega_+}$.
Again, we may assume $c=0$.
Recall that we have set $\Lambda_i\coloneqq \wh{L_i}$ for simplicity of notation.
The following lemma is obtained in \cite[Lem.~6.14]{Gu12}.

\begin{lemma}\label{lem:muhomlocconst}
	Under \pref{as:assumptionclean}, $\mu hom(F_2,F_1)|_{\Omega_+}$ is supported in $\Lambda_1 \cap \Lambda_2$ and has locally constant cohomology sheaves.
\end{lemma}

\begin{proof}
	For completeness, we also give a proof here.
	By \pref{prp:muhomss}, we have
	\begin{equation}\label{eq:SSmuhomF}
	\begin{split}
	\Supp(\mu hom(F_2,F_1)|_{\Omega_+}) & \subset \Lambda_1 \cap \Lambda_2, \\
	\MS(\mu hom(F_2,F_1)|_{\Omega_+}) & \subset -\bfh^{-1}(C( \Lambda_1, \Lambda_2)) \cap T^*\Omega_+.
	\end{split}
	\end{equation}
	Set $\Lambda_{12}\coloneqq \Lambda_1 \cap \Lambda_2$.
	Since $\Lambda_1$ and $\Lambda_2$ intersect cleanly, we have
	\begin{equation}
	C( \Lambda_1, \Lambda_2)
	=
	T\Lambda_1|_{\Lambda_{12}} + T\Lambda_2|_{\Lambda_{12}}.
	\end{equation}
	Since $\Lambda_i$ is Lagrangian, we get $-\bfh^{-1}(T{\Lambda_i}) \subset T^*_{\Lambda_i}T^*(M \times \bR)$ for $i=1,2$.
	In particular, $-\bfh^{-1}(T{\Lambda_i}|_{\Lambda_{12}}) \subset T^*_{\Lambda_{12}}T^*(M \times \bR)$.
	Hence we obtain
	\begin{equation}
	-\bfh^{-1}(C( \Lambda_1, \Lambda_2)) \cap T^*\Omega_+
	\subset T^*_{\Lambda_{12}}T^*(M \times \bR).
	\end{equation}
	Hence by \eqref{eq:SSmuhomF}, $\MS(\mu hom(F_2,F_1)|_{\Omega_+}) \subset T^*_{\Lambda_{12}}T^*(M \times \bR)$, which proves the result.
\end{proof}

Let $C_1,\dots,C_{n_0}$ be the connected components of $L_1 \cap L_2$ with $f_{21}(C_j)=0 \ (j =1,\dots, n_0)$.
For a component $C_j$, we define a closed subset $\wh{C_j}$ of $\Omega_+ \subset T^*(M \times \bR)$ by
\begin{equation}\label{eq:compconification}
\wh{C_j}
\coloneqq 
\{ (x,t;\xi,\tau) \mid \tau >0, (x;\xi/\tau) \in C_j, t=-f_1(x;\xi/\tau) \, (=-f_2(x;\xi/\tau)) \}.
\end{equation}
Note that $\wh{C_j}/\bR_{>0} \simeq C_j$.
We also denote by $d_i \colon \Lambda_i \to \frac{1}{2} \bZ$ the function which assigns the shift of $F_i$.
Since the function $d_i$ is invariant under the $\bR_{>0}$-action, we use the same symbol $d_i$ for the function $ L_i=\Lambda_i / \bR_{>0} \to \frac{1}{2} \bZ$ (see also \pref{sec:appasano}).

\begin{theorem}\label{thm:constF2}
	Under \pref{as:assumptionclean} and in the notation above, assume, moreover, that $\bfk=\bF_2=\bZ/2\bZ$.
	Then 
	\begin{equation}
	\mu hom(F_2,F_1)|_{\Omega_+}
	\simeq
	\bigoplus_{j=1}^{n_0}
	\bfk_{\wh{C_j}}[-s(C_j)],
	\end{equation}
	where $s(C_j) \in \bZ$ is given by
	\begin{equation}\label{eq:defsC}
	s(C_j)
	\coloneqq 
	d_2(p)-d_1(p)
	+\frac{1}{2}(\dim M - \dim C_j)
	-\frac{1}{2} \tau(T_pL_2,T_pL_1,T_p(T^*_{\pi(p)}M))
	\end{equation}
	with $p \in C_j$.
	In particular,
	\begin{equation}
	\RG(\Omega_+;\mu hom(F_2,F_1)|_{\Omega_+})
	\simeq
	\bigoplus_{j=1}^{n_0} \RG(C_j;\bfk_{C_j})[-s(C_j)].
	\end{equation}
\end{theorem}

\begin{proof}
	(a) By \pref{lem:muhomlocconst}, $\mu hom(F_2,F_1)|_{\Lambda_1 \cap \Lambda_2}$ has locally constant cohomology sheaves.
	Fix $p \in C_j$ and let us compute the stalk at $p'\coloneqq (p,0;1) \in \wh{C_j}$.
	There exists a Hamiltonian diffeomorphism with compact support $\psi \colon T^*M \to T^*M$ such that $\psi(L_i)$ is a graph $\Gamma_{d\varphi_i}$ of a $C^\infty$-function $\varphi_i \colon M \to \bR$ in a neighborhood of $\psi(p)$ for $i=1,2$.
	Let $\wh{\psi} \colon \rT(M \times \bR) \to \rT(M \times \bR)$ be the homogeneous Hamiltonian diffeomorphism associated with $\psi$ and $K \in \Db(M \times \bR \times M \times \bR)$ be the sheaf quantization of $\wh{\psi}$.
	For simplicity of notation we set $\chi=\wh{\psi}$.
	By \pref{prp:qctmuhom}, in a neighborhood of $\chi(p')$, we have the isomorphism
	\begin{equation}
	\mu hom(K \circ F_2,K \circ F_1)
	\simeq
	\chi_* \mu hom(F_2,F_1).
	\end{equation}
	Moreover, by \pref{prp:qctshift}, $K \circ F_i$ is simple with shift  $d_i(p)+d'-\delta_i$ along $\chi(\Lambda_i)$ at $\chi(p')$, where $d'$ is the shift of $K$ at $(\chi(p'),p'^a)$ and
	\begin{equation}
	\delta_i
	\coloneqq 
	\frac{1}{2}(\dim M+1)
	+\frac{1}{2}
	\tau \left( \lambda_\infty(p'),\lambda_{\Lambda_i}(p'),\chi^{-1}(\lambda_\infty(\chi(p'))) \right).
	\end{equation}
	Here we use the symbols $\lambda_{\Lambda}(p)$ and $\lambda_\infty(p)$ defined in \eqref{eq:notaionlambda}.
	Hence we obtain the isomorphism $K \circ F_i \simeq \bfk_{N_i}\left[ d_i(p)+d'-\delta_i - \frac{1}{2} \right]$ in $\Db(M \times \bR;\chi(p'))$, where $N_i\coloneqq \{(x,t) \in M \times \bR \mid \varphi_i(x)+t=0\}$ (see also \pref{eg:simpleshift}).
	Thus we get
	\begin{equation}
	\begin{split}
	\mu hom(F_2,F_1)_{p'}
	& \simeq
	\mu hom(\bfk_{N_2},\bfk_{N_1})_{\chi(p')}[d_1(p)-d_2(p)-\delta_1+\delta_2] \\
	& \simeq
	\mu_{N_2}(\bfk_{N_1})_{\chi(p')}[d_1(p)-d_2(p)-\delta_1+\delta_2],
	\end{split}
	\end{equation}
	where we used \pref{prp:muhomprop}(iii) for the second isomorphism.
	We introduce a new local coordinate system $(x,t')$ on $M \times \bR$ by $t'\coloneqq t+\varphi_2(x)$.
	Then $N_2=\{ t'=0 \}$ and
	$N_1=\{t'=\varphi_2(x)-\varphi_1(x) \}$.
	\pref{as:assumptionclean} implies that  $\varphi\coloneqq \varphi_2-\varphi_1$ is a Morse--Bott function.
	Therefore, after changing the local coordinate system $x$ on $M$, we may assume that $\pi(\chi(p'))=(0,0)$ in the coordinates $(x,t')$ and $\varphi(x)=-x_1^2 - \dots -x_\lambda^2+x_{\lambda+1}^2 + \dots +x_l^2$, where $l\coloneqq \dim M -\dim C_j$.
	Note that in the coordinate system on $T^*(M \times \bR)$ associated with $(x,t')$, we have $\chi(p')=(0,0;0,1)$.
	Hence by \eqref{eq:microlocalizationstalk} we obtain
	\begin{equation}
	\begin{split}
	\mu_{N_2}(\bfk_{N_1})_{\chi(p')}
	& \simeq
	\mu_{\bR^{\dim M} \times \{0\}}(\bfk_{\{t'=\varphi(x)\}})_{(0,0;0,1)} \\
	& \simeq
	\RG_{\{t' \ge 0 \}}(\bfk_{\{t'=\varphi(x)\}})_0 \\
	& \simeq
	\bfk[-\lambda].
	\end{split}
	\end{equation}
	Thus $\mu hom(F_2,F_1)|_{\wh{C_j}}$ is concentrated in some degree and locally constant of rank $1$.
	Since $\bfk=\bF_2$, a locally constant sheaf of rank $1$ is constant, which implies the isomorphism $\mu hom(F_2,F_1)|_{\wh{C_j}} \simeq \bfk_{\wh{C_j}}[d_1(p)-d_2(p)-\delta_1+\delta_2-\lambda]$.
	
	\noindent(b)
	We shall prove 
	\begin{equation}
	\lambda
	+\delta_1-\delta_2
	=
	\frac{1}{2}(\dim M - \dim C_j) -\frac{1}{2} \tau(\lambda_{L_2}(p),\lambda_{L_1}(p),\lambda_\infty(p)).
	\end{equation}
	For the above coordinates $x$ on $M$, we set $x'=(x_1,\dots,x_l), x''=(x_{l+1},\dots, x_{m})$ with $m=\dim M$ and denote by $(x;\xi)=(x',x'';\xi',\xi'')$ the associated coordinates on $T^*M$.
	We also denote by $\partial^2_{x,x}\varphi(0)=(\partial^2_{x_j x_k}\varphi(0))_{j,k}$ the Hessian of $\varphi$.
	Then, by a similar argument to that of the proof of \cite[Prop.~7.5.3]{KS90}, we get
	\begin{equation}
	\begin{split}
	\tau(\lambda_\infty(0), T_{0} (T^*_{\bR^m}\bR^m), T_{0}\Gamma_{d\varphi})
	& =
	\tau(\{ x=0 \}, \{ \xi=0 \}, \{ \xi= \partial^2_{x,x}\varphi(0) \cdot x \}) \\
	& =
	\tau(\{ x'=0 \}, \{ \xi'=0 \}, \{ \xi'=\partial^2_{x',x'}\varphi(0) \cdot x'  \}) \\
	& =
	-\sgn (\partial^2_{x',x'}\varphi(0))
	= 2\lambda -l.
	\end{split}
	\end{equation}
	Moreover, we have
	\begin{equation}
	\begin{split}
	& \tau(\chi(\lambda_{\Lambda_2}(p')),\chi(\lambda_{\Lambda_1}(p')),\lambda_{\infty}(\chi(p'))) \\
	= \ &
	\tau(\lambda_{\wh{T^*_{M}M}}(\chi(p')), \lambda_{\wh{\Gamma_{-d\varphi}}}(\chi(p')), \lambda_\infty(\chi(p'))) \\
	= \ &
	\tau(T_{0} (T^*_{\bR^m}\bR^m), T_0\Gamma_{-d\varphi}, \lambda_\infty(0)) \\
	= \ &
	-\tau(\lambda_\infty(0), T_{0} (T^*_{\bR^m}\bR^m), T_{0}\Gamma_{d\varphi}).
	\end{split}
	\end{equation}
	Here we used the homogeneous symplectic coordinate system associated with $(x,t')$ for the first equality, \pref{lem:tauconic} for the second one, and \pref{prp:tauproperties}(i) for the last one.
	Combining the above two equalities, we finally obtain
	\begin{equation}
	\begin{split}
	-2\lambda+l -2\delta_1+2\delta_2
	= \	&
	\tau(\chi(\lambda_{\Lambda_2}(p')),\chi(\lambda_{\Lambda_1}(p')),\lambda_\infty(\chi(p')))
	-2\delta_1+2\delta_2 \\
	= \ &
	\tau(\lambda_{\Lambda_2}(p'),\lambda_{\Lambda_1}(p'),\chi^{-1}(\lambda_\infty(\chi(p')))) \\
	& +\tau(\lambda_{\Lambda_1}(p'),\lambda_\infty(p'),\chi^{-1}(\lambda_\infty(\chi(p')))) \\
	& +\tau(\lambda_\infty(p'),\lambda_{\Lambda_2}(p'),\chi^{-1}(\lambda_\infty(\chi(p')))) \\
	= \ &
	\tau(\lambda_{\Lambda_2}(p'),\lambda_{\Lambda_1}(p'),\lambda_\infty(p')) \\
	= \ &
	\tau(\lambda_{L_2}(p),\lambda_{L_1}(p),\lambda_\infty(p)).
	\end{split}
	\end{equation}
	Here the second equality follows from the invariance under symplectic isomorphisms, the third one follows from the ``cocycle condition" of the inertia index (\pref{prp:tauproperties}(ii)), and the last one follows from \pref{lem:tauconic} again.
	Since $l=\dim M -\dim C_j$, this completes the proof.
\end{proof}

For a general filed~$\bfk$, if $L_1$ and $L_2$ are the graphs of exact 1-forms and intersect cleanly, the locally constant object $\mu hom(F_2,F_1)|_{\Omega_+}$ is described as follows:

\begin{proposition}
	Let $\bfk$ be any field.
	Under \pref{as:assumptionclean}, assume, moreover, that there exists a $C^\infty$-function $\varphi_i \colon M \to \bR$ such that $L_i=\Gamma_{d\varphi_i}$ and $f_i=\varphi_i \circ \pi|_{L_i}$ for $i=1,2$.
	Define a Morse--Bott function $\varphi$ on $M$ by $\varphi\coloneqq \varphi_2-\varphi_1$ and let $C_1,\dots,C_{n_0}$ be the critical components of $\varphi$ with $\varphi(C_j)=0 \ (j =1,\dots, n_0)$.
	For such a critical component $C_j$, define $T_{C_j}^-M$ as the maximal subbundle of $T_{C_j}M$ where the restriction of the Hessian $\mathrm{Hess}(\varphi)|_{T_{C_j}^-M}$ is negative definite, and define a closed subset $\wh{C_j}$ of $\Omega_+$ by
	\begin{equation}
	\wh{C_j}
	\coloneqq 
	\{ (x, -\varphi_1(x);\tau d\varphi_1(x),\tau) \mid \tau>0, x \in C_j \}.
	\end{equation}	
	Moreover, let $\cL_i\coloneqq (F_i)_+ \in \Module(\bfk_M)$ be the locally constant sheaf of rank $1$ associated with the simple sheaf quantization $F_i$ for $i=1,2$.
	Then 
	\begin{equation}
	\begin{split}
	\mu hom(F_2,F_1)|_{\Omega_+}
	& \simeq 
	\bigoplus_{j=1}^{n_0} \pi_j^{-1} \left( \omega_{C_j/T_{C_j}^-M} \otimes \cL_1 \otimes \cL_2^{\otimes -1} \right) \\
	& \simeq 
	\bigoplus_{j=1}^{n_0} \pi_j^{-1} \left( \ori_{C_j/T_{C_j}^-M} \otimes \cL_1 \otimes \cL_2^{\otimes -1} \right) [-s(C_j)],
	\end{split}
	\end{equation}
	where $\pi_j \colon \wh{C_j} \to C_j$ is the projection, $s(C_j) \in \bZ$ is the fiber dimension of $T_{C_j}^-M$, which is equal to $s(C_j)$ given by \eqref{eq:defsC} in the statement of \pref{thm:constF2}, and the right-hand sides denote their zero-extensions to $\Omega_+$ by abuse of notation.
\end{proposition}

\begin{proof}
	We may assume that $\varphi_1 \equiv 0, \varphi_2 \equiv \varphi$ and $\cL_i \simeq \bfk_M$ for $i=1,2$.
	Then $F_1 \simeq \bfk_{M \times [0,+\infty)}$ and $F_2 \simeq \bfk_{\{ (x,t) \mid \varphi(x)+t \ge 0 \}}$.
	Take a critical component $C_j$ of $\varphi$ satisfying $\varphi(C_j)=0$.
	Then by \pref{prp:microlocalcoh} we have
	\begin{equation}
	\begin{split}
	& \mu hom(\bfk_{\{ (x,t) \mid \varphi(x)+t \ge 0 \}},\bfk_{M \times [0,+\infty)})|_{\wh{C_j}} \\
	\simeq \ & 
	\pi_j^{-1} \RG_{\{ (x,t) \mid \varphi(x)+t \ge 0 \}}(\bfk_{M \times [0,+\infty)})|_{C_j \times \{0\}} \\
	\simeq \ &
	\pi_j^{-1} \RG_{\{ (x,t) \mid \varphi(x)+t \ge 0 \}}(\bfk_{M \times \{0\}})|_{C_j \times \{0\}} \\
	\simeq \ & 
	\pi_j^{-1} \RG_{\{ \varphi \ge 0 \}}(\bfk_{M})|_{C_j}.
	\end{split}
	\end{equation}
	Moreover, we obtain (cf.\ \cite[Cor.~1.3]{ST92})
	\begin{equation}
	\RG_{\{ \varphi \ge 0 \}}(\bfk_{M})|_{C_j}
	\simeq 
	\RG_{C_j}(\bfk_{T_{C_j}^-M})|_{C_j}
	\simeq 
	\omega_{C_j/T_{C_j}^-M},
	\end{equation}
	which completes the proof.
\end{proof}

In the case $L_1$ and $L_2$ intersect transversally, we also obtain the following:

\begin{proposition}\label{prp:transverse}
	Let $\bfk$ be any field and assume that $L_1$ and $L_2$ intersect transversally.
	For an intersection point $p \in L_1 \cap L_2$ with $f_2(p)-f_1(p)=0$, define $\wh{p}\coloneqq \{(\tau p,-f_{1}(p);\tau) \in T^*M \times T^*\bR \mid \tau >0 \} \subset \Omega_+$ as a special case of \eqref{eq:compconification}.
	Then
	\begin{equation}
	\mu hom(F_2,F_1)|_{\Omega_+}
	\simeq
	\bigoplus_{\substack{p \in L_1 \cap L_2, \\ f_{2}(p)-f_{1}(p)=0}}\bfk_{\wh{p}}[-s(p)],
	\end{equation}
	where $s(p) \in \bZ$ is given by \eqref{eq:defsC} in the statement of \pref{thm:constF2}.
\end{proposition}

\begin{proof}
	In this case, the support of $\mu hom(F_2,F_1)|_{\Omega_+}$ is contained in $\bigsqcup_p \wh{p}$ and each $\wh{p}$ is contractible.
	Hence $\mu hom(F_2,F_1)|_{\Omega_+}$ has constant cohomology sheaves on $\bigsqcup_p \wh{p}$.
	The rest is exactly the same as the proof of \pref{thm:constF2}.
\end{proof}

The relation between the degree $s(C)$ and the Maslov index will be explored in \pref{sec:appasano}.

\begin{theorem}\label{thm:cleanineq}
	Under \pref{as:assumptionclean}, let $L_1 \cap L_2=\bigsqcup_{j=1}^n C_j$ be the decomposition into connected components.
	Recall that for a component $C$ of $L_1 \cap L_2$, one defines $f_{21}(C)\coloneqq f_2(p)-f_1(p)$, taking some $p \in C$.
	Moreover, let $a,b \in \bR$ with $a<b$ or $a \in \bR, b=+\infty$.
	Then 
	\begin{equation}
	\sum_{a \le f_{21}(C_j) <b} \dim_{\bF_2} H^{k-s(C_j)}(C_j;\bF_2)
	\ge
	 \dim_{\bF_2}
	H^k \RG_{M \times [a,b)}((-\infty,b);\cHom^\star(F_2,F_1))
	\end{equation}
	for any $k \in \bZ$, where $s(C_j)$ is given by \eqref{eq:defsC} in the statement of \pref{thm:constF2}.
	In particular,
	\begin{equation}
	\sum_{j=1}^n \dim_{\bF_2} H^{k-s(C_j)}(C_j;\bF_2) \ge
	\dim_{\bF_2} \Hom_{\cT(M)}(F_2,F_1[k])
	\end{equation}
	for any $k \in \bZ$.
	If $L_1$ and $L_2$ intersect transversally, the inequalities hold for any field $\bfk$, not only for $\bF_2$.
\end{theorem}

\begin{proof}
	Since the set $\{ f_2(p)-f_1(p) \mid p \in L_1 \cap L_2 \} \subset \bR$ is finite,  conditions (1) and (2) in \pref{prp:ineqmuhomver} are satisfied.
	Moreover, by \pref{thm:constF2}, condition~(3) is also satisfied.
	Hence the first assertion follows from \pref{prp:ineqmuhomver} and \pref{thm:constF2}.
	For the second assertion, by \pref{prp:morT} it is enough to show that
	\begin{equation}
	\sum_{j=1}^n \dim_{\bF_2} H^{k-s(C_j)}(C_j;\bF_2)
	\ge
	\dim_{\bF_2} \Hom_{\cD(M)}(F_2,{T_c}_*F_1[k])
	\end{equation}
	for any $c \in \bR$ and any $k \in \bZ$.
	This follows from \pref{prp:morD} and the first assertion for the case $a=0, b=+\infty$.
	The last assertion follows from \pref{prp:transverse}.
\end{proof}

\begin{corollary}[{\cite[Thm.~1.3.1]{Nad} and \cite[Thm.~1]{FSS08}}]
	Under \pref{as:assumptionclean} and in the same notation as in \pref{thm:cleanineq}, one has
	\begin{equation}
	\sum_{j=1}^n \sum_{k \in \bZ} \dim_{\bF_2} H^k(C_j;\bF_2)
	\ge
	\sum_{k \in \bZ} \dim_{\bF_2} H^k(M;\bF_2).
	\end{equation}
	If $L_1$ and $L_2$ intersect transversally, then 
	\begin{equation}
	\#(L_1 \cap L_2) \ge
	\sum_{k \in \bZ} \dim H^k(M;\cL)
	\end{equation}
	for any rank $1$ locally constant sheaf $\cL \in \Module(\bfk_M)$ over any field $\bfk$.
\end{corollary}

\begin{proof}
	It follows from \pref{prp:cohlagr} and \pref{thm:cleanineq}.
\end{proof}

\begin{remark}
	Assume $L_1=L_2=L$ and $f_1=f_2$, and set $\cL_i\coloneqq (F_i)_+$ for $i=1,2$.
	Then $\{\mu hom({T_c}_*F_2,F_1)|_{\Omega_+}\}_c$ is concentrated at $c=0$ and $\mu hom(F_2,F_1)|_{\Omega_+} \simeq \pi_{\wh{L}}^{-1}(\cL_2 \otimes \cL_1^{\otimes -1})$, where $\pi_{\wh{L}} \colon \wh{L} \to M$ is the projection, over any field $\bfk$.
	Let $a,b \in \bR$ with $a<b$ or $a \in \bR, b=+\infty$.
	In this case, we obtain a more precise description of the complex $\RG_{M \times [a,b)}(M \times (-\infty,b);\cHom^\star(F_2,F_1))$, not only the Morse--Bott-type inequality.
	Namely, if $a \le 0 <b$, using the concentration, \pref{lem:microsection}, and \pref{prp:Homstarmuhom}, we have
	\begin{equation}
	\begin{split}
	& \RG_{M \times [a,b)}(M \times (-\infty,b);\cHom^\star(F_2,F_1)) \\
	\simeq \ &
	\RG(M \times \{0\};\RG_{M \times [0,+\infty)}(\cHom^\star(F_2,F_1))|_{M \times \{0\}}) \\
	\simeq \ &
	\RG(\Omega_+;\mu hom(F_2,F_1)|_{\Omega_+}) \\
	\simeq \ &
	\RG \left(\wh{L};\pi_{\wh{L}}^{-1}(\cL_2 \otimes \cL_1^{\otimes -1}) \right).
	\end{split}
	\end{equation}
	This is essentially one of the results of Guillermou \cite[Thm.~20.4]{Gu12}.
\end{remark}

\appendix

\section{Degenerate Lagrangian intersections}\label{sec:nonclean}

In this section, using very simple examples, we briefly remark that our method can also deal with degenerate Lagrangian intersections.
Until the end of this section we set $\bfk=\bQ$. 
We shall consider $T^*S^1$ and the intersection of the zero-section $S^1$ and the graph of an exact 1-form $L=\Gamma_{df}$.
Let $F\coloneqq \bfk_{S^1 \times [0,+\infty)}$ be the canonical sheaf quantization associated with the zero-section $S^1$ and $G\coloneqq \bfk_{\{ (x,t) \in S^1 \times \bR \mid f(x)+t \ge 0 \}}$ be that associated with $L$.
Assume that the intersection of $S^1$ and $L$ has only one possibly degenerate component $C$ and it is transversal outside $C$.
Then by \pref{prp:ineqmuhomver} and a similar argument to the proof of \pref{thm:cleanineq} we obtain
\begin{equation}\label{eq:ineqnonclean}
\begin{split}
& \#\{ p \in S^1 \cap L \mid \text{$p$ is a transverse intersection point} \} \\
& + \sum_{k \in \bZ} \dim H^k \RG \left(\Omega_+ \cap \pi^{-1}(C);\mu hom(F,G)|_{\Omega_+ \cap \pi^{-1}(C)} \right) \\
& \qquad \ge
\sum_k \dim \Hom_{\cT(S^1)}(F,G[k])
=\sum_k \dim H^k(S^1;\bfk_{S^1})=2.
\end{split}
\end{equation}
We calculate the ``contribution" $\RG \left(\Omega_+ \cap \pi^{-1}(C);\mu hom(F,G)|_{\Omega_+ \cap \pi^{-1}(C)} \right)$ from $C$ in the following two typical examples.

First we consider the case that the intersection is as in \pref{fg:lag1} in a neighborhood of $C$.
In this case, $G$ is isomorphic to the constant sheaf supported in the shaded closed subset in \pref{fg:sq1} in a neighborhood of $C$.
\begin{figure}[H]
	\begin{minipage}{0.49\hsize}
		\begin{center}
			\begin{tikzpicture}[scale=1.2]
			\draw[->] (-2,0) -- (2,0) node[below] {$x$};
			\draw[->] (0,-1) -- (0,1) node[left]  {$\xi$};
			\draw[thick,-] plot[domain=-2:-1, variable=\x, smooth] ({\x},{(\x +1)^2});
			\draw[ultra thick,-] (-1,0) -- (1,0);
			\draw[thick,-] plot[domain=1:2, variable=\x, smooth] ({\x},{-(\x -1)^2});
			\path (-1.8,1) node[below right] {$L$};
			\path (-1,0) node[below] {$a$};
			\path (1,0) node[below] {$b$};
			\path (0,0) node[below right] {$C$};
			\draw (1,2pt) -- (1,-2pt);
			\draw (-1,2pt) -- (-1,-2pt);
			\end{tikzpicture}
			\caption{$L$ in the first example}
			\label{fg:lag1}
		\end{center}
	\end{minipage}
	\begin{minipage}{0.49\hsize}
		\begin{center}
			\begin{tikzpicture}[scale=1.2]
			\fill [gray!50] plot[domain=-2:-1, variable=\x, smooth] ({\x},{-(\x+1)^3})--(1,0)--plot[domain=1:2, variable=\x, smooth] ({\x},{(\x-1)^3})--(-2,1);
			\draw[->] (-2.2,0) -- (2.2,0) node[below] {$x$};
			\draw[->] (0,-0.7) -- (0,1.3) node[left]  {$t$};
			\draw[thick,-] plot[domain=-2:-1, variable=\x, smooth] ({\x},{-(\x+1)^3});
			\draw[thick,-] (-1,0) -- (1,0);
			\draw[thick,-] plot[domain=1:2, variable=\x, smooth] ({\x},{(\x-1)^3});
			\path (-1,0) node[below] {$a$};
			\path (1,0) node[below] {$b$};
			\draw (1,2pt) -- (1,-2pt);
			\draw (-1,2pt) -- (-1,-2pt);
			\end{tikzpicture}
			\caption{$G$ in the first example}
			\label{fg:sq1}
		\end{center}
	\end{minipage}
\end{figure}
\noindent Hence we find that $\mu hom(F,G)|_{\Omega_+ \cap \pi^{-1}(C)} \simeq \bfk_{[a,b] \times (0,+\infty)}$ and
\begin{equation}
\RG(\Omega_+ \cap \pi^{-1}(C);\mu hom(F,G)|_{\Omega_+ \cap \pi^{-1}(C)})
\simeq
\RG([a,b];\bfk_{[a,b]})
\simeq \bfk.
\end{equation}
Thus in this case the contribution from $C$ is 1 in \eqref{eq:ineqnonclean}, and the cardinality of the transverse intersection points is at least 1, as expected.

Next we consider the case that the intersection is as in \pref{fg:lag2} in a neighborhood of $C$.
The canonical sheaf quantization $G$ associated with $L$ is isomorphic to the constant sheaf supported in the shaded closed subset in \pref{fg:sq2} in a neighborhood of $C$.
\begin{figure}[H]
	\begin{minipage}{0.49\hsize}
		\begin{center}
			\begin{tikzpicture}[scale=1.2]
			\draw[->] (-2,0) -- (2,0) node[below] {$x$};
			\draw[->] (0,-0.8) -- (0,1.5) node[left]  {$\xi$};
			\draw[thick,-] plot[domain=-2:-1, variable=\x, smooth] ({\x},{(\x +1)^2});
			\draw[ultra thick,-] (-1,0) -- (1,0);
			\draw[thick,-] plot[domain=1:2, variable=\x, smooth] ({\x},{(\x -1)^2});
			\path (1.4,1) node[below right] {$L$};
			\path (-1,0) node[below] {$a$};
			\path (1,0) node[below] {$b$};
			\path (0,0) node[below right] {$C$};
			\draw (1,2pt) -- (1,-2pt);
			\draw (-1,2pt) -- (-1,-2pt);
			\end{tikzpicture}
			\caption{$L$ in the second example}
			\label{fg:lag2}
		\end{center}
	\end{minipage}
	\begin{minipage}{0.49\hsize}
		\begin{center}
			\begin{tikzpicture}[scale=1.2]
			\fill [gray!50] plot[domain=-2:-1, variable=\x, smooth] ({\x},{-(\x+1)^3})--(1,0)--plot[domain=1:2, variable=\x, smooth] ({\x},{-(\x-1)^3})--(2,1)--(-2,1);
			\draw[->] (-2.2,0) -- (2.2,0) node[below] {$x$};
			\draw[->] (0,-1) -- (0,1.3) node[left]  {$t$};
			\draw[thick,-] plot[domain=-2:-1, variable=\x, smooth] ({\x},{-(\x+1)^3});
			\draw[thick,-] (-1,0) -- (1,0);
			\draw[thick,-] plot[domain=1:2, variable=\x, smooth] ({\x},{-(\x-1)^3});
			\path (-1,0) node[below] {$a$};
			\path (1,0) node[below] {$b$};
			\draw (1,2pt) -- (1,-2pt);
			\draw (-1,2pt) -- (-1,-2pt);
			\end{tikzpicture}
			\caption{$G$ in the second example}
			\label{fg:sq2}
		\end{center}
	\end{minipage}
\end{figure}
\noindent Therefore, in this case we get $\mu hom(F,G)|_{\Omega_+ \cap \pi^{-1}(C)} \simeq \bfk_{[a,b) \times (0,+\infty)}$ and
\begin{equation}
\RG(\Omega_+ \cap \pi^{-1}(C);\mu hom(F,G)|_{\Omega_+ \cap \pi^{-1}(C)})
\simeq
\RG_c([a,b);\bfk_{[a,b)})
\simeq 0.
\end{equation}
Hence the contribution from $C$ is 0 in \eqref{eq:ineqnonclean} and the cardinality of the transverse intersection points is at least $2$ in the second case.

\begin{remark}
	For $i=1,2$, let $L_i$ be a compact connected exact Lagrangian submanifold and $f_i \colon L_i \to \bR$ be a function satisfying $df_i=\alpha|_{L_i}$.	
	Moreover, let $F_i$ be a simple sheaf quantization associated with $L_i$ and $f_i$.
	\pref{prp:ineqmuhomver} says that the contribution from components on which $f_2(p)-f_1(p)=c$ is encoded in the sheaf $\mu hom({T_c}_*F_2,F_1)|_{\Omega_+}$ (even for possibly degenerate Lagrangian intersections).
	If the intersection is clean along a component $C$, then $\mu hom({T_c}_*F_2,F_1)|_{\Omega_+}$ is locally constant of rank $1$ on the cone of $C$ as in \pref{lem:muhomlocconst}.
	However, as seen in the above examples, if the intersection is degenerate, then $\mu hom({T_c}_*F_2,F_1)|_{\Omega_+}$ is not necessarily locally constant.
\end{remark}

\section{Functoriality of sheaf quantizations}\label{sec:appfunctorial}

In this section we prove the "functoriality" of Guillermou's simple sheaf quantizations with respect to Hamiltonian diffeomorphisms.
We remark that results in this section are independent of the results in \pref{sec:intersection} and not used for the proofs of them.

Let $L$ be a compact connected exact Lagrangian submanifold of $T^*M$ and $f$ be a primitive of the Liouville form $\alpha$.
We define the conification $\wh{L}_f$ of $L$ with respect to $f$ as in \eqref{eq:conification}.
Let $\psi$ be a Hamiltonian diffeomorphism of $T^*M$ and $\phi=(\phi_s)_s \colon T^*M \times I \to T^*M$ be a Hamiltonian isotopy, where $I$ is an open interval containing $[0,1]$, such that $\phi_1=\psi$ and $\phi_0=\id_{T^*M}$.
We denote by $H=(H_s)_s \colon T^*M \times I \to \bR$ the associated Hamiltonian and by $X_s$ the associated Hamiltonian vector field on $T^*M$.
The homogeneous lift $\wh{\phi}$ of $\phi$ is described as follows (see \cite[Prop.~A.6]{GKS}):
\begin{equation}
\wh{\phi}_1(x,t;\xi,\tau)=(x',t+u(x;\xi/\tau);\xi',\tau),
\end{equation}
where $(x';\xi'/\tau)=\phi_1(x;\xi/\tau)=\psi(x;\xi/\tau)$ and $u \colon T^*M \to \bR$ is defined by
\begin{equation}
	u(p)=\int_0^1 (H_s-\alpha(X_s))(\phi_s(p)) \, ds.
\end{equation}
Hence we get
\begin{equation*}
\begin{split}
\wh{\phi}_1(\wh{L}_f)
&=
\left\{
(x',t+u(x;\xi/\tau);\xi',\tau) \; \middle| \;
\begin{split}
& \tau >0, \exists (x;\xi) \ \text{ s.t.}\ (x;\xi/\tau) \in L, \\
& (x';\xi'/\tau)=\psi(x;\xi/\tau), t=-f(x;\xi/\tau)
\end{split}
\right\} \\
&=
\left\{
(x',t';\xi',\tau) \; \middle| \;
\begin{split}
& \tau >0, (x';\xi'/\tau) \in \psi(L), \\
& t'=-f \circ \psi^{-1}(x';\xi'/\tau)+u \circ \psi^{-1}(x';\xi'/\tau)
\end{split}
\right\}.
\end{split}
\end{equation*}
On the other hand, we have equalities
\begin{align*}
\psi^*\alpha-\alpha
&=
\int_0^1 \left(\frac{d}{ds} \phi_s^*\alpha\right) \, ds \\
&=
\int_0^1 \phi_s^* (L_{X_s}\alpha) \, ds \\
&=
\int_0^1 \phi_s^* (d\iota_{X_s}\alpha+\iota_{X_s}d\alpha) \, ds \\
&=
d \int_0^1 \phi_s^* (\alpha(X_s)-H_s)\, ds
=-du.
\end{align*}
Here, for a vector field $X$, $L_X$ denotes the Lie derivative with respect to $X$, and the third equality follows from Cartan's formula.
Moreover, the fourth equality follows from the definition of the Hamiltonian vector field: $d\alpha(X_s,\ast)=-dH_s$.
Hence setting $\tl f\coloneqq (f-u) \circ \psi^{-1} \colon \psi(L) \to \bR$, we get
\begin{equation}
\begin{split}
\alpha|_{\psi(L)}
&=(\psi^{-1})^* (\alpha|_L -du|_L) \\
&=(\psi^{-1})^* (df -du|_L)
=d \tl{f}.
\end{split}
\end{equation}
Thus we find that $\tl{f}$ is a primitive of $\alpha$ on $\psi(L)$ and obtain the following:

\begin{lemma}\label{lem:conicham}
	One has 
	\begin{equation}\label{eq:conicham}
	\wh{\phi}_1(\wh{L}_f)=\wh{\psi(L)}_{\tl{f}}
	\ \subset T^*(M \times \bR).
	\end{equation}
\end{lemma}

\begin{proposition}
	Let $\cL \in \Module(\bfk_M)$ be a locally constant sheaf of rank $1$ and $F_{L}$ be the simple sheaf quantization of $\wh{L}_f$ satisfying ${F_{L}}_+ \simeq \cL$.
	Let $\psi \colon T^*M \to T^*M$ be a Hamiltonian diffeomorphism and $\Psi \colon \Db(M \times \bR) \to \Db(M \times \bR)$ the associated functor (see \eqref{eq:functorPsi}).
	Define $\tl f\coloneqq (f-u) \circ \psi^{-1} \colon \psi(L) \to \bR$ as above and denote by $\wh{\psi(L)}_{\tl f}$ the conification of $\psi(L)$ with respect to $\tl f$.
	Moreover, let $F_{\psi(L)}$ be the simple sheaf quantization of $\wh{\psi(L)}_{\tl f}$ satisfying ${F_{\psi(L)}}_+ \simeq \cL$.
	Then 
	\begin{equation}
	\Psi(F_{L}) \simeq F_{\psi(L)}.
	\end{equation}
\end{proposition}

\begin{proof}
	By \pref{lem:conicham}, we have
	\begin{equation}
	\Psi(F_{L}) \in \Db_{\wh{\psi(L)}_{\tl{f}} \, \cup T^*_{M \times \bR}(M \times \bR)}(M \times \bR).
	\end{equation}
	By the uniqueness of simple sheaf quantizations (\pref{thm:Guiquan}),
	it remains to show that
	\begin{equation}
	\Psi(F_{L})_- \simeq 0, \quad
	\Psi(F_{L})_+ \simeq \cL.
	\end{equation}
	Let $\wh{\phi} \colon \rT(M \times \bR) \times I \to \rT(M \times \bR)$ be the associated homogeneous Hamiltonian isotopy and $K \in \Dlb(M \times \bR \times M \times \bR \times I)$ be the sheaf quantization of $\wh{\phi}$.
	Let $\varepsilon \in \bR_{>0}$ satisfying $[-\varepsilon,1+\varepsilon] \subset I$.
	By the compactness of $L$, there exists $A \in \bR_{>0}$ satisfying
	\begin{equation}
	\bigcup_{s \in [-\varepsilon,1+\varepsilon]} \wh{\phi}_s(\wh{L}_f)
	\subset T^*(M \times (-A,A)).
	\end{equation}
	Replacing $I$ with the relatively compact subinterval $(-\varepsilon,1+\varepsilon)$, we may assume that
	\begin{equation}\label{eq:condham}
	\bigcup_{s \in I} \wh{\phi}_s(\wh{L}_f) \subset T^*(M \times (-A,A))
	\end{equation}
	and $K \in \Db(M \times \bR \times M \times \bR \times I)$ from the beginning.
	Set $G\coloneqq (K \circ F_{L})|_{M \times (A,+\infty) \times I} \in \Db(M \times (A,+\infty) \times I)$.
	We shall show that
	\begin{equation}\label{eq:SSG}
	\MS(G) \subset T^*_{M \times (A,+\infty) \times I}(M \times (A,+\infty) \times I).
	\end{equation}
	First, by \pref{prp:SScomp}, we have
	\begin{equation}
	\MS(K \circ F_{L})
	\subset
	(\Lambda_{\wh{\phi}} \circ \wh{L}_f) \cup T^*_{M \times \bR \times I}(M \times \bR \times I).
	\end{equation}
	By the definition of $\Lambda_{\wh{\phi}}$ (see \eqref{eq:deflambdahatphi}), we obtain
	\begin{equation}\label{eq:est1}
	(\Lambda_{\wh{\phi}} \circ \wh{L}) \cap (T^*_{M \times \bR}(M \times \bR) \times T^* I) 
	\subset
	T^*_{M \times \bR \times I}(M \times \bR \times I).
	\end{equation}
	Denote by $i_s \colon M \times \bR \times \{s\} \hookrightarrow M \times \bR \times I$ the closed embedding for any $s \in I$.
	Then, by the definition of $\Lambda_{\wh{\phi}}$, we also have
	\begin{equation}\label{eq:est2}
	{(i_s)}_d {(i_s)}_\pi^{-1}(\Lambda_{\wh{\phi}} \circ \wh{L}_f)
	=\wh{\phi}_s(\wh{L}_f).
	\end{equation}
	Moreover by \eqref{eq:condham}, we get
	\begin{equation}\label{eq:est3}
	\wh{\phi}_s(\wh{L}_f) \cap T^*(M \times (A,+\infty))
	=\emptyset
	\end{equation}
	for any $s \in I$.
	Hence the inclusion \eqref{eq:SSG} follows from the above estimates~\eqref{eq:est1}, \eqref{eq:est2}, and \eqref{eq:est3}.
	Since $I$ is contractible, we have $G \simeq q^{-1}(G|_{M \times (A,+\infty) \times \{0\}})$,	where $q \colon M \times (A,+\infty) \times I \to M \times (A,+\infty)$ is the projection.
	In particular, we get
	\begin{align*}
	\Psi(F_{L})|_{M \times (A,+\infty)}
	& =
	G|_{M \times (A,+\infty) \times \{1\}} \\
	& \simeq
	G|_{M \times (A,+\infty) \times \{0\}} \\
	& \simeq
	(F_{L})|_{M \times (A,+\infty)}
	\simeq \cL \boxtimes \bfk_{(A,+\infty)}
	\end{align*}
	and $\Psi(F_{L})_+ \simeq \cL$.
	A similar argument shows that $\Psi(F_{L})_- \simeq 0$.
\end{proof}

\section{Relation to grading in Lagrangian Floer cohomology theory, by Tomohiro Asano}\label{sec:appasano}

In this section we relate the absolute grading of $\cHom^\star$ to that of Lagrangian Floer cohomology.

\subsection{Inertia index and Maslov index}

In this subsection we recall some properties of the inertia index and the Maslov index.
First we list some properties of the inertia index.

\begin{proposition}[{\cite[Thm.~A.3.2]{KS90}}]\label{prp:tauproperties}
	Let $E$ be a symplectic vector space and denote by $\cL(E)$ the Lagrangian Grassmannian of $E$. 
	The inertia index $\tau \colon \cL(E)^3\to \bZ$ satisfies the following properties.  
	\begin{enumerate}
		\item For any $\lambda_1,\lambda_2,\lambda_3 \in \cL(E)$,  
		$\tau(\lambda_1,\lambda_2,\lambda_3)
		=-\tau(\lambda_2,\lambda_1,\lambda_3)
		=-\tau(\lambda_1,\lambda_3,\lambda_2)$. 
		\item The inertia index satisfies the ``cocycle condition": for any quadruple $\lambda_1,\lambda_2,\lambda_3,\lambda_4 \in \cL(E)$, 
		\begin{equation} 
		\tau(\lambda_1,\lambda_2,\lambda_3)
		=\tau(\lambda_1,\lambda_2,\lambda_4) 
		+\tau(\lambda_2,\lambda_3,\lambda_4)
		+\tau(\lambda_3,\lambda_1,\lambda_4).
		\end{equation}
		\item If $\lambda_1,\lambda_2,\lambda_3$ move continuously in the Lagrangian Grassmannian $\cL(E)$ so that \linebreak $\dim(\lambda_1 \cap \lambda_2), \allowbreak \dim(\lambda_2 \cap \lambda_3), \allowbreak \dim(\lambda_3 \cap \lambda_1)$ remain constant, then $\tau(\lambda_1,\lambda_2,\lambda_3)$ remains constant.
		\item Let $E'$ be another symplectic vector space, and let $\lambda_1,\lambda_2,\lambda_3$ (resp.\ $\lambda'_1,\lambda'_2,\lambda'_3$) be a triple of Lagrangian subspaces of $E$ (resp.\ $E'$).
		Then 
		\begin{equation}
		\tau_{E \oplus E'}(\lambda_1 \oplus \lambda'_1, \lambda_2 \oplus \lambda'_2, \lambda_3 \oplus \lambda'_3)
		=
		\tau_{E}(\lambda_1,\lambda_2,\lambda_3)
		+
		\tau_{E'}(\lambda'_1,\lambda'_2,\lambda'_3). 
		\end{equation}		
	\end{enumerate}
\end{proposition}

Let $M$ be a compact connected manifold without boundary and $T^*M$ be its cotangent bundle.
Moreover, let $\cL_{T^*M}$ be the fiber bundle over $T^*M$ whose fiber is the Lagrangian Grassmannian, that is, $\cL_{T^*M,p}=\cL(T_pT^*M)$.
Denote by $\lambda_\infty \colon T^*M \to \cL_{T^*M}, p \mapsto T_pT_{\pi(p)}^*M$ be the section which assigns the fiber to $p$. 
A Lagrangian submanifold $L$ of $T^*M$ defines a section $\lambda_L \colon L \to \cL_{T^*M}, p \mapsto T_pL$
over $L$.

\begin{lemma}\label{lem:tauconic}
	For $i=1,2$, let $L_i$ be a compact connected exact Lagrangian submanifold and $f_i \colon L_i \to \bR$ be a function such that $df_i=\alpha|_{L_i}$ and set $\Lambda_i\coloneqq \wh{L_i}_{f_i}$, the conification of $L_i$ with respect to $f_i$. 	
	Let $p \in L_1 \cap L_2$ and assume $f_1(p)=f_2(p)$.
	Set $p'\coloneqq (p,-f_1(p);1) \in \Lambda_1 \cap \Lambda_2 \subset T^*(M \times \bR)$.
	Then 	
	\begin{equation}
	\tau_{T_{p'}T^*(M \times \bR)}(\lambda_{\Lambda_2}(p'),\lambda_{\Lambda_1}(p'),\lambda_\infty(p'))
	=
	\tau_{T_pT^*M}(\lambda_{L_2}(p),\lambda_{L_1}(p),\lambda_\infty(p)).
	\end{equation}
\end{lemma}

\begin{proof}
	Take a local homogeneous symplectic coordinate system $(x,t;\xi,\tau)$ on $T^*(M \times \bR)$.
	Using the coordinate system, we identify $T_{p'}T^*(M \times \bR)$ with $\bR^m \times \bR \times \bR^m \times \bR$.	
	In this coordinate system, we get $\lambda_\infty(p')=0 \times 0 \times \bR^m \times \bR$.
	Write $p=(x;\xi)$ using the coordinate system.
	Then $\lambda_{\Lambda_i}(p')$ is spanned by 
	\begin{equation}
	(0,0;\xi,1), 
	(v_i,-Tf_i(v_i);\zeta_i,0) \ \left((v_i,\zeta_i) \in T_{p}L_i \right).
	\end{equation}
	For $r\in [0,1]$, let $\lambda_{\Lambda_i}(p';r)$ be the Lagrangian linear subspace spanned by 
	\begin{equation}
	(0,0;r\xi,1), 
	(v_i,-r \cdot Tf_i(v_i);\zeta_i,0) \ \left((v_i,\zeta_i) \in T_{p}L_i \right).
	\end{equation}
	Then by \pref{prp:tauproperties}(iii) we have
	\begin{equation}
	\tau_{T_{p'}T^*(M \times \bR)}(\lambda_{\Lambda_2}(p'),\lambda_{\Lambda_1}(p'),\lambda_\infty(p'))
	=
	\tau_{T_{p'}T^*(M \times \bR)}(\lambda_{\Lambda_2}(p';r),\lambda_{\Lambda_1}(p';r),\lambda_\infty(p'))
	\end{equation}
	for any $r \in [0,1]$.
	Since $\lambda_{\Lambda_i}(p';0)=\lambda_{L_i}(p)\oplus \bR\langle (0 ;1) \rangle$, by \pref{prp:tauproperties}(iv), we obtain
	\begin{equation}
	\begin{aligned}
	& \tau_{T_{p'}T^*(M \times \bR)}(\lambda_{\Lambda_2}(p'),\lambda_{\Lambda_1}(p'),\lambda_\infty(p')) \\
	= \ & 
	\tau_{T_{p'}T^*(M \times \bR)}(\lambda_{\Lambda_2}(p';0),\lambda_{\Lambda_1}(p';0),\lambda_\infty(p')) \\
	= \ &
	\tau_{T_pT^*M}(\lambda_{L_2}(p), \lambda_{L_1}(p), \lambda_\infty(p)).
	\end{aligned}
	\end{equation}
\end{proof}

Next we recall some properties of the Maslov index (see, for example, Leray~\cite{Leray}, Robbin--Salamon~\cite{RS93}, and de~Gosson~\cite{deGosson}). 

\begin{proposition}\label{prp:muproperties}
	Let $E$ be a symplectic vector space and denote by $\tl{\cL}(E)$ the universal covering of the Lagrangian Grassmannian $\cL(E)$ of $E$. For $\tl{\lambda_i}\in \tl{\cL}(E)  (i\in \bN)$, denote its projection to $\cL(E)$ by $\lambda_i$. 
	The Maslov index $\mu \colon \tl{\cL}(E)^2\to \frac{1}{2}\bZ$ satisfies the following properties.  
	\begin{enumerate}
		\item For any $\tl{\lambda_1}, \tl{\lambda_2} \in\tl{\cL}(E)$, 
		$\mu (\tl{\lambda_1}, \tl{\lambda_2})= -\mu (\tl{\lambda_2}, \tl{\lambda_1})$
		\item The coboundary of $\mu$ is given by $\tau$ : 
		$\mu(\tl{\lambda_1},\tl{\lambda_2})
		+\mu(\tl{\lambda_2},\tl{\lambda_3})
		+\mu(\tl{\lambda_3},\tl{\lambda_1})
		=
		\frac{1}{2}
		\tau(\lambda_1,\lambda_2,\lambda_3)$
		\item If $\tl{\lambda_1}$ and $\tl{\lambda_2}$ move continuously in $\tl{\cL} (E)$ so that $\dim(\lambda_1 \cap \lambda_2)$ remains constant, then $\mu (\tl{\lambda_1},\tl{\lambda_2})$ remains constant. 
		\item For any $\tl{\lambda_1}, \tl{\lambda_2} \in\tl{\cL}(E)$, 
		$\mu (\tl{\lambda_1}, \tl{\lambda_2})\equiv \frac{1}{2}(\dim (\lambda_1\cap\lambda_2)+\frac{1}{2}\dim E) \mod \bZ$. 
		\item Under an isomorphism $\rho \colon \pi_{1}(\cL (E))\simeq \bZ $, 
		for any $\tl{\lambda_1}, \tl{\lambda_2} \in\tl{\cL}(E)$ and $n, m\in \bZ$, 
		$\mu (\rho^{-1}(n)\cdot \tl{\lambda_1}, \rho^{-1}(m) \cdot \tl{\lambda_2})=\mu (\tl{\lambda_1}, \tl{\lambda_2})+n-m$, 
		where dots stand for the covering transformation.
	\end{enumerate}
\end{proposition}

\begin{remark}
	Notation for the Maslov index differs between authors. 
	Our $\mu$ is equal to half of the $\mu$ in \cite{deGosson}. 
	Note that (ii) and (iii) of the above proposition determine the function $\mu \colon \tl{\cL}(E)^2\to \frac{1}{2}\bZ$ uniquely. 
\end{remark}

\subsection{Graded Lagrangian submanifolds and Maslov index}

Next we recall the notion of graded Lagrangian submanifolds due to Seidel~\cite{Seidel00}. 
Denote by $\tl{\cL}_{T^*M}$ the fiberwise universal cover of $\cL_{T^*M}$ whose fiber over $p$ is identified with the space of the homotopy classes of paths in $\cL_{T^*M,p}$ from $\lambda_\infty$.
We also denote by  $\mu \colon \tl{\cL}_{T^*M}\times_{T^*M} \tl{\cL}_{T^*M} \to \frac{1}{2} \bZ$ the Maslov index on $T^\ast M$. 
For a Lagrangian submanifold $L$ of $T^*M$, a \emph{grading} of $L$ is a lift $\tl{\lambda} \colon L \to \tl{\cL}_{T^*M}$ of $\lambda_L$.
A \emph{graded Lagrangian submanifold} is a pair $(L,\tl{\lambda})$ consisting of a Lagrangian submanifold $L$ and a grading $\tl{\lambda}$ of $L$.
\begin{equation}
\begin{split}
\xymatrix@C=40pt{
	& \tl{\cL}_{T^*M} \ar[d]^-{\bZ} \\
	& \cL_{T^*M} \ar[d] \\
	L \ \ar@<-0.4ex>@{^{(}->}[r] \ar[ru]_-{\lambda_L} \ar@{-->}[ruu]^-{\tl{\lambda}} & T^*M \ar@/_/[u]_-{\lambda_\infty}
}
\end{split}
\end{equation}

Now, let $(L_1,\tl{\lambda_1})$ and $(L_2,\tl{\lambda_2})$ be graded Lagrangian submanifolds of $T^*M$ intersecting cleanly.
For a connected component $C$ of $L_1 \cap L_2$, we define the absolute grading $\gr(L_2,L_1;C)$ of $C$ by 
taking $p\in C$ and 
\begin{equation}
\gr(L_2,L_1;C)
=
\frac{1}{2} (\dim M -\dim C)- \mu(\tl{\lambda_2}(p),\tl{\lambda_1}(p)),
\end{equation}
which induces the absolute grading of Lagrangian Floer cohomology. 
Note that by \pref{prp:muproperties} (i) and (ii), the grading $\gr(L_2,L_1;C)$ is written as 
\begin{equation}\label{eq:maslovgr}
\begin{split}
& \gr(L_2,L_1;C) \\
& =  
\frac{1}{2} (\dim M -\dim C)
+\mu(\tl{\lambda_1}(p),\lambda_\infty(p))+\mu(\lambda_\infty(p),\tl{\lambda_2}(p))
-\frac{1}{2} \tau(\lambda_2(p),\lambda_1(p),\lambda_\infty(p)) \\
& = 
\frac{1}{2} (\dim M -\dim C)
+\mu(\lambda_\infty(p),\tl{\lambda_2}(p))-\mu(\lambda_\infty(p),\tl{\lambda_1}(p))
-\frac{1}{2} \tau(T_pL_2,T_pL_1,\lambda_\infty(p)),
\end{split}
\end{equation}
where the point $\lambda_\infty(p)$ is regarded as (the homotopy class of) the constant path.

\subsection{Shifts of simple sheaf quantizations}

Let $L$ be a compact exact Lagrangian submanifold of $T^*M$ and $f \colon L \to \bR$ be a primitive of the Liouville $1$-form. 
Denote by $\wh{L} \subset T^*(M \times \bR)$ the conification of $L$ with respect to $f$ and let $F \in \Db(M \times\bR)$ be a simple sheaf quantization of $\wh{L}$.
By \pref{thm:Guiquan}, the object $F$ is simple along $\wh{L}$ and the shift of $F$ at a point of $\wh{L}$ defines a function $d \colon \wh{L} \to \frac{1}{2} \bZ$.
Since $d(c \cdot p')=d(p')$ for any $p' \in \wh{L}$ and $c \in \bR_{>0}$, and $\wh{L}/\bR_{>0}=L$, we also regard $d$ as a function $L \to \frac{1}{2} \bZ$.

\begin{proposition}\label{prp:grading}
	There is a grading $\tl{\lambda} \colon L \to \tl{\cL}_{T^*M}$ such that 
	\begin{equation}\label{eq:grading}
	\mu(\lambda_\infty(p),\tl{\lambda}(p))+\frac{1}{2} (\dim M +1) =d(p),
	\end{equation}
	where $\lambda_\infty$ denotes the constant path.
\end{proposition}

\begin{proof}
	Let $U_{L} \subset \cL_{T^*M}|_{L}$ be the open subset of the Lagrangian Grassmannian restricted over $L$ consisting of Lagrangian subspaces transversal to $\lambda_\infty$ and $\lambda_{L}$.
	Moreover, let $U \subset U_{L}$ be a connected open subset of $U_{L}$ 
	whose image $\pi (U)$ under $\pi$ is contractible.  
	Note that the set of  such $\pi(U)$ covers $L$. 
	For $p \in L$, we set $p'\coloneqq (p,-f(p);1) \in \wh{L}$. 
	Fix a local section $\gamma \colon \pi(U) \to U$ and take a local section $\gamma' \colon \rho^{-1}(\pi(U))\to \cL_{T^*(M \times \bR)}|_{\wh{L}}$ 
	so that $\gamma'(p')=\gamma(p)\oplus \bR\langle (1;0)\rangle \subset T_p T^* M \oplus T_{(-f(p); 1)} T^* \bR$ holds for every $p\in \pi(U)$. 
	By \pref{prp:muproperties} and the same homotopy $\lambda_{\wh{L}}(p'; r)$ as in the proof of \pref{lem:tauconic}, 
	we get 
	\begin{equation}
	\frac{1}{2} \tau(\lambda_\infty(p'),\lambda_{\wh{L}}(p'),\gamma'(p'))
	=
	\mu(\lambda_\infty(p),\tl{\lambda}(p))
	+\mu(\tl{\lambda}(p),\tl{\gamma}(p))
	+\mu(\tl{\gamma}(p),\lambda_\infty(p)),
	\end{equation}
	where $\tl{\gamma}$ and $\tl{\lambda}$ are locally defined lifts of $\gamma$ and $\lambda_L$. 
	Since the image of $\gamma$ is contained in a connected component of $U_{L}$, both $\mu(\tl{\lambda}(p),\tl{\gamma}(p))$ and $\mu(\tl{\gamma}(p),\lambda_\infty(p))$ are constant on $\pi(U)$.
	The difference between the shifts can be calculated as 
	\begin{equation}
	\begin{split}
	d(p)-d(q)
	& =
	\frac{1}{2}
	\left(
	\tau(\lambda_\infty(p'),\lambda_{\wh{L}}(p'),\gamma(p'))
	-
	\tau(\lambda_\infty(q'),\lambda_{\wh{L}}(q'),\gamma(q'))
	\right) \\
	& =
	\mu(\lambda_\infty(p),\tl{\lambda}(p))-\mu(\lambda_\infty(q),\tl{\lambda}(q)) 
	\end{split}
	\end{equation}
	(see \cite[Section~8]{Gu12}). 
	Hence the function $d(p)-\mu(\lambda_\infty(p),\tl{\lambda}(p))$ is constant on $\pi(U)$ with value in $\frac{1}{2} \bZ$. 
	Moreover, since $\mu(\tl{\lambda}(p),\gamma(p))\equiv \mu(\gamma(p),\lambda_\infty(p))\equiv \frac{1}{2}\dim M \mod \bZ$, we have 
	\begin{equation}
	d(p)-\mu(\lambda_\infty(p),\tl{\lambda}(p))
	\equiv \frac{1}{2} \dim(M \times \bR)=\frac{1}{2} (\dim M +1) \mod \bZ.
	\end{equation}
	By \pref{prp:muproperties}(v), $\tl{\lambda}$ can be  uniquely chosen so that \pref{eq:grading} holds on $\pi (U)$. Such $\tl{\lambda}$ can be glued together on the whole of $L$ and becomes a grading of $L$. 
\end{proof}

Next we consider the degree of $\cHom^{\star}(F_2,F_1)$.
Let $L_1$ and $L_2$ be compact exact Lagrangian submanifolds of $T^*M$ intersecting cleanly.
For $i=1,2$, take a primitive $f_i \colon L_i \to \bR$ of the Liouville 1-form and denote by $\wh{L_i}$ the conification of $L_i$ with respect to $f_i$. 
Let $F_i \in \Db(M \times \bR)$ be a simple sheaf quantization of $\wh{L_i}$.
We also denote by $d_i \colon L_i \to \frac{1}{2} \bZ$ the function which assigns the shift of $F_i$.
Then by \pref{thm:constF2}, the degree associated with a component $C$ of $L_1 \cap L_2$ in $\cHom^\star(F_2,F_1)$ is given by 
\begin{equation}\label{eq:degreeshift}
d_2(p)-d_1(p) 
+\frac{1}{2}(\dim M - \dim C) 
-\frac{1}{2} \tau(T_pL_2,T_pL_1,\lambda_\infty(p))
\end{equation}
for any $p \in C$.
Thus, combining \pref{prp:grading} with \eqref{eq:maslovgr} and \eqref{eq:degreeshift}, we obtain the following theorem.

\begin{theorem}
	For $i=1,2$, let $\tl{\lambda_i} \colon L_i \to \tl{\cL}_{T^*M}$ be the grading of $L_i$ given in \pref{prp:grading}.
	Then the degree associated with a component $C$ of $L_1 \cap L_2$ in $\cHom^\star(F_2,F_1)$ is equal to $\gr(L_2,L_1;C)$.
\end{theorem}

\def\cprime{$'$}

\end{document}